\documentclass[12pt]{article}
\usepackage[utf8]{inputenc}
\usepackage{amsthm}
\usepackage{amsfonts}
\usepackage{amsmath}
\usepackage{colonequals}
\usepackage{txfonts}
\usepackage{url}
\newtheorem{theorem}{Theorem}

\newtheorem{lemma}[theorem]{Lemma}
\newtheorem{corollary}[theorem]{Corollary}
\newtheorem{observation}[theorem]{Observation}
\newcommand{\inter}{\text{int}}
\newcommand{\ext}{\text{ext}}
\newcommand{\out}{\text{out}}
\newcommand{\supp}{\text{supp}}
\newcommand{\circul}{\text{circ}}
\newcommand{\dual}[1]{#1^\star}
\newcommand{\gain}[1]{\text{slack}_{#1}}
\newcommand{\PP}{\mathcal{P}}
\newcommand{\CC}{\mathcal{C}}
\newcommand{\XX}{\mathcal{X}}

\newcommand{\RR}{\mathcal{R}}

\title{Coloring near-quadrangulations of the cylinder and the torus\thanks{Research supported by project 17-04611S (Ramsey-like aspects of graph
coloring) of Czech Science Foundation.}}

\author{Zdeněk Dvořák\thanks{Charles University, Prague, Czech Republic.  E-mail: {\tt rakdver@iuuk.mff.cuni.cz}.}\and
Jakub Pek\'arek\thanks{Charles University, Prague, Czech Republic.  E-mail: {\tt pekarej@iuuk.mff.cuni.cz}.}}

\begin{document}
\maketitle

\begin{abstract}
Let $G$ be a simple connected plane graph and let $C_1$ and $C_2$ be cycles in $G$ bounding distinct
faces $f_1$ and $f_2$.  For a positive integer $\ell$, let
$r(\ell)$ denote the number of integers $n$ such that $-\ell\le n\le \ell$, $n$ is divisible by $3$, and $n$ has the same
parity as $\ell$; in particular, $r(4)=1$.  Let $r_{f_1,f_2}(G)=\prod_f r(|f|)$, where the product is over the faces $f$
of $G$ distinct from $f_1$ and $f_2$, and let $q(G)=1+\sum_{f:|f|\neq 4} |f|$, where the sum is over all faces $f$ of $G$.
We give an algorithm with time complexity $O\bigl(r_{f_1,f_2}(G)q(G)|G|\bigr)$ which, given
a $3$-coloring $\psi$ of $C_1\cup C_2$, either finds an extension of $\psi$ to a $3$-coloring of $G$,
or correctly decides no such extension exists.

The algorithm is based on a min-max theorem for a variant of integer 2-commodity flows, and
consequently in the negative case produces an obstruction to the existence of the extension.
As a corollary, we show that every triangle-free graph drawn in the torus with edge-width at least $21$ is $3$-colorable.
\end{abstract}

While it is NP-hard to decide whether a planar graph is $3$-colorable~\cite{garey1979computers},
every planar triangle-free graph is $3$-colorable~\cite{grotzsch1959}.  This fact motivated
the development of a rich theory of $3$-colorability of triangle-free embedded graphs.

Much of our knowledge about colorability of embedded graphs comes from the study of critical graphs that can
be drawn in a given surface.
A graph is \emph{$k$-critical} if its chromatic number is $k$, but every proper subgraph of $G$ is $(k-1)$-colorable.
Clearly, a graph is $(k-1)$-colorable if and only if it does not contain a $k$-critical subgraph.
Gallai~\cite{galfor} gave a lower bound on the density of $k$-critical graphs.  Together with the generalized Euler formula,
this bound implies that there exist only finitely many $4$-critical graphs of girth at least six that can be embedded in
any fixed surface $\Sigma$.  Consequently, it is possible to test $3$-colorability
of a graph of girth at least six embedded in $\Sigma$ by testing the presence of finitely many obstructions.
A deep result of Thomassen~\cite{thomassen-surf} shows this is the case for graphs of girth at least five
as well.  Furthermore, all graphs of girth at least five drawn in the projective plane, the torus~\cite{thom-torus},
or the Klein bottle~\cite{tw-klein} are $3$-colorable.

For triangle-free graphs, the situation is much more complicated.  Already in the projective plane,
there are infinitely many triangle-free $4$-critical graphs---by a result of Gimbel and Thomassen~\cite{gimbel},
these are exactly the non-bipartite quadrangulations of the projective plane not containing separating $4$-cycles.
Substantially generalizing this result, Dvořák, Král' and Thomas~\cite{trfree4} proved that $4$-critical triangle-free
graphs embedded in a fixed surface $\Sigma$ \emph{without non-contractible 4-cycles} are near-quadrangulations,
in the sense that there exists a constant $c_\Sigma$ such that all faces of such graph have length at most $c_\Sigma$
and all but $c_\Sigma$ faces have length~$4$.  A detailed treatment of $4$-critical triangle-free graphs with
non-contractible $4$-cycles was given by Dvořák and Lidický~\cite{cylgen-part3}.

A common theme in almost all mentioned results is the need to deal with the precoloring extension variant of
the problem, where vertices incident with a bounded number of faces are given a fixed coloring and we need to decide
whether this coloring extends to a $3$-coloring of the whole graph.  Indeed, the basis for the results mentioned
in the previous paragraph is the following theorem.
\begin{theorem}[Dvořák, Král' and Thomas~\cite{trfree2}]
Let $G$ be a simple plane graph of girth at least five, with the outer face bounded by a cycle $C$, and let $\psi$ be a $3$-coloring of $C$.
If $\psi$ does not extend to a $3$-coloring of $G$, but extends to a $3$-coloring of every proper subgraph of $G$ which contains $C$,
then $|G|\le 1715|C|$.
\end{theorem}
Let us remark that by $|G|$ and $\Vert G\Vert$ we denote the number of vertices and edges of $G$, respectively;
the graphs we consider are allowed to have parallel edges and loops, unless they are specified to be simple.
The constant $1715$ was later improved to $37/3$, even in the list coloring setting~\cite{dk}.
Based on this result, if the length of the precolored cycle $C$ is bounded by a constant, then it is possible
to decide whether a precoloring of $C$ extends by testing the presence of finitely many obstructions.
A more involved argument~\cite{dvokawalg} shows that even when the length of $C$ is not bounded, there is a polynomial-time
algorithm for the precoloring extension problem from a cycle in a planar graph of girth at least 5.

For triangle-free graphs, Dvořák, Král' and Thomas~\cite{trfree7} designed a polynomial-time algorithm
to test $3$-colorablity of a triangle-free graph embedded in a fixed surface, and more generally,
to solve the precoloring extension problem with precolored vertices incident with a bounded number of faces
of bounded length.  The algorithm is based on the aforementioned fact that $4$-critical triangle-free graphs without
non-contractible $4$-cycles are near-quadrangulations~\cite{trfree4},
and consequently its key part is an algorithm for the precoloring extension problem in near-quadrangulations~\cite{trfree6}.
However, this algorithm is not practical due to large multiplicative constants.

Dvořák and Lidický~\cite{col8cyc} found a practical algorithm for the precoloring extension problem from a single cycle
in planar near-quadrangulations.  
For a positive integer $\ell$, let $r(\ell)$ denote the number of integers $n$ such that $-\ell\le n\le \ell$, $n$ is divisible
by $3$, and $n$ has the same parity as $\ell$.  For a plane graph $G$ and some of its faces $f_1,\ldots, f_t$, let $r_{f_1,\ldots,f_t}(G)=\prod_f r(|f|)$, where the product is over the faces $f$
of $G$ distinct from $f_1$, \ldots, $f_t$.  Let $q(G)=1+\sum_{f:|f|\neq 4} |f|$, where the sum is over all faces $f$ of $G$.
\begin{theorem}[{Dvořák and Lidický~\cite[Lemma~4 and the discussion following it]{col8cyc}}]\label{thm-onec}
There exists an algorithm which, given a simple connected plane graph $G$, a cycle $C_1$ bounding a face $f_1$ of $G$,
and a $3$-coloring $\psi$ of $C_1$, finds in time $O\bigl(r_{f_1}(G)q(G)|G|\bigr)$ a $3$-coloring of $G$ extending $\psi$ or correctly decides
that no such $3$-coloring exists.
\end{theorem}
Note that $r(4)=1$ and $r(\ell)\le \Big\lceil \tfrac{2\ell + 1}{6}\Big\rceil$.  Hence, if $G$ is a near-quadrangulation, 
then the algorithm from Theorem~\ref{thm-onec} has time complexity $O(|G|)$.
The algorithm is practical, only requiring one run of the max-flow algorithm in the plane dual of $G$.
As our main result, we generalize Theorem~\ref{thm-onec} to graphs with two precolored cycles.

\begin{theorem}\label{thm-mainalg}
There exists an algorithm which, given a simple connected plane graph $G$, cycles $C_1$ and $C_2$ bounding distinct faces $f_1$ and $f_2$ of $G$,
and a $3$-coloring $\psi$ of $C_1\cup C_2$, finds in time $O\bigl(r_{f_1,f_2}(G)q(G)|G|\bigr)$ a $3$-coloring of $G$ extending $\psi$ or correctly decides
that no such $3$-coloring exists.
\end{theorem}

The algorithm is again practical, without large multiplicative constants in the time complexity.
The case of planar graphs with two precolored cycles is relevant to the problem of $3$-coloring graphs drawn in
the torus. Let $G$ be a graph drawn in the torus and let $C$ be a non-contractible cycle in $G$.  Cut the
torus along $C$ and patch the resulting holes by disks, to obtain 
a plane graph $G'$ with distinct faces bounded by cycles $C_1$ and $C_2$ of length $|C|$.
Then, $G$ is $3$-colorable if and only if there exists a $3$-coloring $\psi$ of $C$ such that the corresponding
$3$-coloring of $C_1\cup C_2$ extends to a $3$-coloring of $G'$.  Since there are less than $2^{|C|-2}$
distinct $3$-colorings of $C$ (up to permutation of colors), we obtain the following algorithm for coloring
near-quadrangulations of torus with bounded edge-width.  A drawing of a graph in a surface is \emph{$2$-cell}
if each face is homeomorphic to the open disk, and the \emph{edge-width} of an embedded graph is the
length of a shortest non-contractible cycle in the drawing.

\begin{corollary}\label{cor-torus1}
There exists an algorithm which, given a simple graph $G$ with a $2$-cell drawing in the torus of edge-width at most $k$,
finds in time $O(2^kr(G)(k+q(G))|G|)$ a $3$-coloring of $G$ or correctly decides that no such $3$-coloring exists.
\end{corollary}

As we mentioned before, $4$-critical triangle-free graphs without non-contractible $4$-cycles embedded in a fixed
surface are near-quadrangulations~\cite{trfree4}.  Actually, for the torus, the assumption of absence of non-contractible $4$-cycles
can be dropped.  Dvořák and Pekárek~\cite{dvopek} proved the following stronger result regarding toroidal graphs.
For a graph $G$ with a $2$-cell embedding in a surface, let $S(G)$ denote the multiset of the lengths of the faces
of $G$ of length other than $4$.  We call the multisets $\emptyset$, $\{5,5\}$, $\{5,7\}$, $\{5,5,6\}$, and $\{5,5,5,5\}$ \emph{torus-realizable}.
\begin{theorem}[Dvořák and Pekárek~\cite{dvopek}]\label{thm-dvopek}
Let $G$ be a triangle-free graph embedded in the torus.  If $G$ is $4$-critical, then $S(G)$ is torus-realizable, and in particular
$r(G)\le 16$ and $q(G)\le 21$.
\end{theorem}

The algorithm from Theorem~\ref{thm-mainalg} is based on a min-max theorem for a variant of integer 2-commodity flows, and
consequently in the negative case produces an obstruction to the existence of the extension.  Using Theorem~\ref{thm-dvopek}
and an examination of this obstruction, we prove the following.
\begin{lemma}\label{lemma-ew}
If $G$ is a $4$-critical triangle-free graph drawn in the torus, then the edge-width of the drawing is at most $20$.
\end{lemma}
Consequently, we have the following.
\begin{corollary}\label{cor-ew}
Every triangle-free graph embedded in the torus with edge-width at least $21$ is $3$-colorable.
\end{corollary}
Note that Corollary~\ref{cor-ew} is a special case of a more general result proved (with a much larger bound on the edge-width) in~\cite{trfree4}.
Furthermore, let us remark that the bound on the edge-width can be further
improved. Indeed, in~\cite{dpfuture}, we give a computer-assisted argument showing that a triangle-free graph embedded in the torus with edge-width at least $5$ is $3$-colorable,
unless it contains a specific $13$-vertex $4$-critical quadrangulation discovered in~\cite{thomas2008coloring}
as a subgraph.

The rest of the paper is organized as follows.  In Section~\ref{sec-flow}, we give a min-max result on a variant of integer 2-commodity flows.
In Section~\ref{sec-cyl}, we apply this result to prove Theorem~\ref{thm-mainalg}.  Finally, in Section~\ref{sec-tor}, we prove Lemma~\ref{lemma-ew}.

\section{Flows and disjoint non-contractible cycles}\label{sec-flow}

For a plane graph $H$, let $V(H)$, $E(H)$, and $F(H)$ denote the sets of its vertices, edges, and faces, respectively.
For a function $d:X\to \mathbf{Z}$, let $d(X)=\sum_{x\in X} d(x)$.
Let $H$ be a connected plane graph and let $d:V(H)\to \mathbf{Z}$ be a function such that $d(V(H))=0$.
Throughout the paper, we view paths in graphs as directed, i.e., having distinguished starting and ending vertices.
A \emph{$d$-linkage} in $H$ is a system $\PP$ of pairwise edge-disjoint paths in $H$
such that
\begin{itemize}
\item if a path $P\in\PP$ starts in a vertex $u$ and ends in a vertex $v$, then $d(u)>0$ and $d(v)<0$, and
\item every vertex $v\in V(H)$ is the starting or ending vertex of exactly $|d(v)|$ paths from $\PP$.
\end{itemize}
Let $s$ and $t$ be distinct faces of $H$.  A cycle $C$ in $H$ is \emph{$(s,t)$-non-contractible} if the open disk in the plane bounded by $C$ contains
exactly one of the faces $s$ and $t$, and \emph{$(s,t)$-contractible} otherwise.  An \emph{$(s,t)$-circulation} is a set of pairwise edge-disjoint $(s,t)$-non-contractible cycles.
The \emph{support} of a $d$-linkage or an $(s,t)$-circulation $\XX$ is the set $\supp(\XX)$ of edges of $\bigcup \XX$.
The goal of this section is to answer the following question: Given the graph $H$, its faces $s$ and $t$, and the function $d$ as described,
what is the maximum size of an $(s,t)$-circulation whose support is disjoint from the support of some $d$-linkage in $H$?  We will not be able to give a full answer, however we will solve the special case arising in the context
of Theorem~\ref{thm-mainalg}, where the function $d$ satisfies the condition that $\deg_H(v)$ and $d(v)$ have the same parity for all $v\in V(H)$.

It is more convenient to work in the setting of integral flows.  Let us start by giving the necessary notation and describing
their relationship to linkages and circulations.
Consider an arbitrary orientation $\vec{H}$ of $H$.  For a vertex $v\in V(H)$, let $N_{\vec{H}}^-(v)$ and $N_{\vec{H}}^+(v)$ denote the sets
of edges entering and leaving $v$, respectively. We drop the $\vec{H}$ subscripts when the directed graph $\vec{H}$ is
clear from the context. An (integral) \emph{$d$-flow} in $H$ is a function $h:E\bigl(\vec{H}\bigr)\to\{-1,0,1\}$ such that
$\sum_{e\in N^+(v)} h(e)-\sum_{e\in N^-(v)} h(e)=d(v)$ for all $v\in V(H)$; i.e., $d(v)$ gives the amount of the flow originating in $v$.
The \emph{support} $\supp(h)$ of the flow is the set of edges of $H$ such that $h$ is non-zero on the corresponding directed edge of $\vec{H}$.

Let $\dual{H}$ denote the plane dual of $H$.  For a vertex $v\in V(H)$, a face $f\in F(H)$, or an edge $e\in E(H)$, let
$\dual{v}$, $\dual{f}$, or $\dual{e}$ denote the corresponding face, vertex, or edge of $\dual{H}$, respectively.
For an orientation $\vec{H}$ of $H$ and an edge $e\in E\bigl(\vec{H}\bigr)$, the edge $\dual{e}$ of $\dual{\vec{H}}$
is directed so that, looking along the edge $e$ in its direction, the edge $\dual{e}$ crosses $e$ from left to right.
For a function $d:V(H)\to \mathbf{Z}$ or $h:E\bigl(\vec{H}\bigr)\to\mathbf{Z}$, let $\dual{h}:F(\dual{H})\to \mathbf{Z}$
or $\dual{h}:E\bigl(\dual{\vec{H}}\bigr)\to\mathbf{Z}$ be the corresponding function
such that $\dual{d}(\dual{v})=d(v)$ for every $v\in V(H)$ and $\dual{h}(\dual{e})=h(e)$ for every $e\in E\bigl(\vec{H}\bigr)$.  

Let $G$ be a connected plane graph and let $\vec{G}$ be its orientation.  For a (possibly closed) walk $Q$ in $G$ and an edge
$a\in E\bigl(\vec{G}\bigr)$, let $\sigma(Q,a)$ denote the number of times $Q$ traverses $a$ in the direction of $a$
minus the number of times $Q$ traverses $a$ in the direction opposite to $a$.  In particular, if $G=\dual{H}$,
$\vec{G}=\dual{\vec{H}}$, and $Q$ is a path or a cycle, then for every $e\in E\bigl(\vec{H}\bigr)$, we have
$\sigma(Q,\dual{e})=1$ if $e$ crosses $Q$ from right to left (as seen along the direction of $Q$), $\sigma(Q,\dual{e})=-1$
if $e$ crosses $Q$ from left to right, and $\sigma(Q,\dual{e})=0$ if $e$ does not cross $Q$.
For a function $p:E\bigl(\vec{G}\bigr)\to\mathbf{Z}$, we define $\int_Q p=\sum_{e\in E\bigl(\vec{G}\bigr)} p(e)\sigma(Q,e)$.
Let $\varnothing:V(H)\to \mathbf{Z}$ denote the function whose value is $0$ everywhere.

It is well-known that a flow can be expressed as a disjoint union of cycles and source-sink paths.
Conversely, we can send a unit of flow along each of the paths and in any direction along each of the cycles to obtain
a flow.  Hence, the following claim holds.

\begin{observation}\label{obs-flicir}
Let $H$ be a connected plane graph and let $d:V(H)\to \mathbf{Z}$ be a function such that $d(V(H))=0$.
Let $Q$ be a path in $\dual{H}$ from a vertex $\dual{s}$ to a vertex~$\dual{t}$.
\begin{itemize}
\item If $H$ contains a $d$-flow $h$, then $H$ contains a $d$-linkage $\PP$ with $\supp(\PP)\subseteq \supp(h)$.
If $H$ contains a $d$-linkage $\PP$, then $H$ contains a $d$-flow $h$ with $\supp(h)=\supp(\PP)$.
\item For a non-negative integer $k$, if $H$ contains a $\varnothing$-flow $h$ with $\int_Q \dual{h}=k$, then $H$ contains an $(s,t)$-circulation $\CC$
of size $k$ with $\supp(\CC)\subseteq \supp(h)$.
If $H$ contains an $(s,t)$-circulation $\CC$ of size $k$, then for every integer $k'\in\{-k,-k+2,\ldots,k\}$,
$H$ contains a $\varnothing$-flow $h$ with $\int_Q \dual{h}=k'$ and $\supp(h)=\supp(\CC)$.
\end{itemize}
\end{observation}

Let us now relate the values $\int_{Q_1} \dual{h}$ and $\int_{Q_2} \dual{h}$ for a $d$-flow $h$ in $H$ and distinct paths
$Q_1$ and $Q_2$ in its dual with the same starting and ending vertices.  For a closed walk $R$ in a plane graph $G$ and
a face $f$ of $G$, let $\omega_R(f)$ denote the winding number of $R$ around $f$.
Recall the winding number is defined as follows: Let $p$ be any half-line with the starting point inside $f$ which intersects $G$ only in
edges.  Then $\omega_R(f)$ is equal to the number of times $R$ intersects $p$ from left to right, minus the number of times $R$ intersects $p$ from right to left.
Let us remark that the value of $\omega_R(f)$ is independent of the exact choice of the half-line $p$.
Indeed, the winding number (for $G=\dual{H}$) can also be defined by the following properties.
\begin{observation}\label{obs-wind}
Let $\vec{H}$ be an orientation of a connected plane graph $H$ and let $R$ be a closed walk in $\dual{H}$.
\begin{itemize}
\item[(a)] The outer face $f_0$ of $\dual{H}$ satisfies $\omega_R(f_0)=0$.
\item[(b)] For any edge $e=(u,v)\in E(\vec{H})$, $\omega_R(\dual{u})=\omega_R(\dual{v})+\sigma(R,\dual{e})$.
\end{itemize}
\end{observation}
For a closed walk $R$ in a plane graph $G$ and a function $b:F(G)\to\mathbf{Z}$,
let us define $$\nabla(R,b)=\sum_{f\in F(G)} \omega_R(f)b(f).$$

\begin{lemma}\label{lemma-flowover}
Let $H$ be a connected plane graph, let $d:V(H)\to \mathbf{Z}$ be a function such that $d(V(H))=0$,
and let $h$ be a $d$-flow in $H$.  For any closed walk $R$ in $\dual{H}$, we have $\int_R \dual{h}=\nabla(R,\dual{d})$,
and in particular $\int_R \dual{h}$ is the same for all $d$-flows.  Furthermore,
if $\dual{s}$ and $\dual{t}$ are vertices of $\dual{H}$, $Q_1$ and $Q_2$ are paths from $\dual{s}$ to $\dual{t}$ in $\dual{H}$, and
$R$ is the closed walk obtained as the concatenation of $Q_1$ with the reversal of $Q_2$,
then $\int_{Q_1} \dual{h}=\int_{Q_2} \dual{h}+\nabla(R,\dual{d})$.
\end{lemma}
\begin{proof}
Let $\vec{H}$ be the orientation of $H$ with respect to which the $d$-flow $h$ is defined.
Using Observation~\ref{obs-wind}(b), we have
\begin{align*}
\int_R \dual{h}&=\sum_{e\in E(\vec{H})} \sigma(R,\dual{e})h(e)
=\sum_{e=(v_1,v_2)\in E(\vec{H})} \bigl(\omega_R(\dual{v_1})-\omega_R(\dual{v_2})\bigr)h(e)\\
&=\sum_{v\in V(H)} \omega_R(\dual{v})\Biggl[\sum_{e\in N_{\vec{H}}^+(v)} h(e)-\sum_{e\in N_{\vec{H}}^-(v)} h(e)\Biggr]
=\sum_{v\in V(H)} \omega_R(\dual{v})d(v)\\
&=\sum_{f\in F(\dual{H})} \omega_R(f)\dual{d}(f)=\nabla(R,\dual{d}),
\end{align*}
as required.  Furthermore, if $R$ is the concatenation of $Q_1$ and the reversal of $Q_2$,
then $\int_{Q_1} \dual{h}-\int_{Q_2} \dual{h}=\int_R \dual{h}=\nabla(R,\dual{d})$.
\end{proof}

Let $H$ be a connected plane graph and let $d:V(H)\to \mathbf{Z}$ be a function such that $d(V(H))=0$.
Let $s$ and $t$ be faces of $H$, and let $Q$ be a path in $\dual{H}$ from the vertex $\dual{s}$ to the vertex $\dual{t}$.
A $d$-flow $h$ in $H$ is \emph{$(s,t)$-circulation-maximum} or \emph{$(s,t)$-circulation-minimum} if
$\int_Q \dual{h}$ has the maximum or the minimum possible value, respectively,
among all $d$-flows.  In view of Lemma~\ref{lemma-flowover}, these notions are independent of the choice of the path $Q$.
We now give a min-max condition for circulation-maximum and minimum flows.
A function $d:V(H)\to \mathbf{Z}$ is \emph{even} if $d(v)$ has the same parity as the degree of $v$ for every vertex $v$ of $H$.

\begin{lemma}\label{lemma-circmax}
Let $H$ be a connected plane graph, let $d:V(H)\to \mathbf{Z}$ be a function such that $d(V(H))=0$, and let
$s$ and $t$ be faces of $H$.
Let $h_1$ and $h_2$ be $(s,t)$-circulation-maximum and minimum $d$-flows in $H$, respectively, with maximal supports.
Then there exist paths $Q_1$ and $Q_2$ in $\dual{H}$ from $\dual{s}$ to $\dual{t}$ such that $\int_{Q_1} \dual{h_1}=|E(Q_1)|$ and $\int_{Q_2} \dual{h_2}=-|E(Q_2)|$.
Furthermore $E(H)\setminus \supp(h_1)$ and $E(H)\setminus\supp(h_2)$ are formed by edge-sets of forests in $H$,
and in particular if $d$ is even, then $\supp(h_1)=\supp(h_2)=E(H)$.
\end{lemma}
\begin{proof}
We prove the claims for $h_1$; the claims for $h_2$ follow by the same argument applied to the $(t,s)$-circulation-maximum $d$-flow $h_2$.
Let $Q$ be any path from $\dual{s}$ to $\dual{t}$ in $\dual{H}$.

If there existed a cycle $C$ with $E(C)\subseteq E(H)\setminus \supp(h_1)$, then note that there exists a $\varnothing$-flow $c$ with support $E(C)$ 
such that $\int_Q \dual{c}\ge 0$, obtained by sending one unit of flow along $C$ in an appropriate direction.
Then $h_1+c$ is a $d$-flow which contradicts either the assumption
that $h_1$ is $(s,t)$-circulation-maximum (when $\int_Q \dual{c}>0$) or that $\supp(h_1)$ is maximal among $(s,t)$-circulation-maximum flows.
Hence, $E(H)\setminus \supp(h_1)$ is an edge-set of a forest in $H$.  If $d$ is even, then each vertex is incident with
even number of edges of $E(H)\setminus \supp(h_1)$, and thus the forest cannot have any leaf; this is only possible if $E(H)\setminus \supp(h_1)=\emptyset$.

Hence, we only need to find a path $Q_1$ from $\dual{s}$ to $\dual{t}$ such that $\int_{Q_1} \dual{h_1}=|E(Q_1)|$.
Let $\vec{H}$ be the orientation of $H$ with respect to which the $d$-flow $h_1$ is defined.
Let us form an auxiliary directed graph $\vec{F}$ with the vertex set $V(H)$ and edge set defined as follows.
For each edge $e\in E\bigl(\vec{H}\bigr)$,
\begin{itemize}
\item if $h_1(e)=1$, then we include in $\vec{F}$ the reversal of $e$,
\item if $h_1(e)=-1$, then we include in $\vec{F}$ the edge $e$, and
\item if $h_1(e)=0$, then we include both $e$ and its reversal in $\vec{F}$.
\end{itemize}
The graph $\vec{F}$ is drawn in the plane in the same way as $H$, with the two opposite edges in the
last case drawn close next to each other.  Let us remark that $\vec{F}$ is defined so that
sending a unit of flow along any directed cycle in $\vec{F}$ and adding the flow to $h_1$ results in a valid $d$-flow in $H$.

Let $S$ be the set of faces $x\in F(H)$ for which there exists a path $Q_x$ in $\dual{H}$ from $\dual{s}$ to $\dual{x}$
such that all edges of $\vec{F}$ intersect $Q_x$ from left to right.  If $t\in S$, then by the definition of $\vec{F}$
we conclude that $\int_{Q_t} \dual{h_1}=|E(Q_t)|$, and thus we can set $Q_1=Q_t$.

Hence, suppose that $t\not\in S$.  Let $H_S$ be the graph obtained from $H$ by deleting all vertices and edges
that are only incident with faces in $S$, and let $s'$ be the face of $H_S$ containing $s$.  Since $t\not\in S$,
the boundary of the face $s'$ contains an $(s,t)$-non-contractible cycle $C$.  All edges of $C$ separate a face in $S$
from a face not in $S$, and thus all of them are in $\vec{F}$ directed from right to left (looking from $s'$).
Let $c:E\bigl(\vec{H}\bigr)\to \{-1,0,1\}$ be defined as follows: For each edge $e\in E(C)$ corresponding to an edge $e'\in E\bigl(\vec{H}\bigr)$,
we let $c(e')=1$ if $e$ and $e'$ are directed in the same way and $c(e')=-1$ otherwise.  On all other edges $e'\in E\bigl(\vec{H}\bigr)$, we set $c(e)=0$.
Then $c$ is a $\varnothing$-flow with $\int_Q \dual{c}=1$, and $-1\le h_1(e)+c(e)\le 1$ for every $e\in E\bigl(\vec{H}\bigr)$ by the definition of the graph $\vec{F}$.
Hence, $h_1+c$ is a $d$-flow with $\int_Q \dual{(h_1+c)}>\int_Q \dual{h_1}$, contradicting the assumption that $h_1$ is $(s,t)$-circulation-maximum.
\end{proof}

The proof of Lemma~\ref{lemma-circmax} can be easily turned into an efficient algorithm to find the flows.
For a function $d:V(H)\to \mathbf{Z}$, let $|d|$ denote $1+\sum_{v\in V(H)} |d(f)|$.

\begin{lemma}\label{lemma-circmax-alg}
Let $H$ be a connected plane graph, let $d:V(H)\to \mathbf{Z}$ be an even function such that $d(V(H))=0$, and let
$s$ and $t$ be distinct faces of $H$.  There exists an algorithm with time complexity $O(|d|\cdot\Vert H\Vert)$ which
finds $(s,t)$-circulation-maximum and minimum $d$-flows in $H$ with support $E(H)$, or decides that $H$ contains no $d$-flow.
\end{lemma}
\begin{proof}
It suffices to find an $(s,t)$-circulation-maximum $d$-flow, an $(s,t)$-circulation-minimum $d$-flow is obtained using the
same algorithm to find a $(t,s)$-circulation-maximum $d$-flow.  We can assume that $t$ is the outer face of $H$.
It is possible to find a $d$-flow $h$ in $H$ (or to decide that no $d$-flow exists) in time $O(|d|\cdot \Vert H\Vert)$
using Ford-Fulkerson algorithm.  Since $d$ is even, the complement of $\supp(h)$ is Eulerian; we can partition
it into cycles in linear time and add an arbitrary $\varnothing$-flow on each of the cycles to $h$.  Hence, from now
on $h$ is a $d$-flow with support $E(H)$.

In time $O(\Vert H\Vert)$ we construct the auxiliary graph $\vec{F}$ as described
in the proof of Lemma~\ref{lemma-circmax} (note that $\vec{F}$ is an orientation of $H$, since
$\supp(h)=E(H)$).  We initialize $s'=s$ and $\vec{B}$ as the subgraph of $\vec{F}$
drawn in the boundary of the face $s'$.  Then, we repeat the following steps:
\begin{itemize}
\item[(i)] While there exists an edge $e$ of $\vec{B}$ directed from left to right as seen from $s'$,
\begin{itemize}
\item if $e$ is a bridge in $\vec{B}$ (i.e., it is incident with $s'$ from both sides), then delete $e$ from $\vec{B}$, otherwise
\item let $W$ be the facial walk of the face of $\vec{F}$ not contained in $s'$, delete $e$ from $\vec{B}$, and add $W-e$ to $\vec{B}$.
\end{itemize}
\item[(ii)] If $s'$ became the outer face of $\vec{B}$, then stop; otherwise, remove from $\vec{B}$ all the edges not incident with its outer face,
and all isolated vertices.
\item[(iii)] Note that $\vec{B}$ became a cycle in the previous step (since all its edges are incident with both $s'$ and the outer face).
Add two units of flow on the edges of the cycle $\vec{B}$ to $H$, and reverse the edges in $\vec{F}$ and $\vec{B}$.
\end{itemize}
Note that the part (i) can be implemented efficiently by maintaining a queue of edges directed from left to right,
updated whenever $W-e$ is added to $\vec{B}$ in step (i) or when the direction of edges are altered in step (iii).
Since each edge is added to $\vec{B}$ at most twice (for each face of $\vec{F}$ it is incident with) and deleted at most once,
the total time complexity of (i) over the whole run of the algorithm is $O(\Vert H\Vert)$.  As for the step (iii),
note that the direction of each edge is changed at most once,
and thus the total time complexity of (iii) over the whole run of the algorithm is also $O(\Vert H\Vert)$.
\end{proof}

Let $H$ be a connected plane graph and let $d:V(H)\to \mathbf{Z}$ be a function such that $d(V(H))=0$.
For a cycle $C$ in $\dual{H}$, we define $\dual{\inter}(C)$ to be the set of faces of $\dual{H}$ contained in the open
disk in the plane bounded by $C$, we let $\inter(C)$ be the set of the corresponding vertices of $H$,
and we let $d(C)=\bigl|\sum_{v\in \inter(C)} d(v)\bigr|$.
Note that if $d$ is even, then $|C|\equiv d(C)\pmod 2$ for every cycle $C$ in $\dual{H}$.
Let $\dual{\ext}(C)=F(\dual{H})\setminus \inter(C)$ and $\ext(C)=V(H)\setminus\inter(C)$.
For an edge $a$ of $\dual{H}$, we define $\inter(a)=\dual{\inter}(a)=\emptyset$, $d(a)=0$, and $|a|=2$
(i.e., $a$ is viewed as a cycle of length two).
For a cycle or edge $x$, we let $\gain{d}(x)=|x|-d(x)$.  
The max-flow min-cut theorem implies the following.
\begin{observation}\label{obs-exflow}
Let $H$ be a connected plane graph and let $d:V(H)\to \mathbf{Z}$ be a function such that $d(V(H))=0$.
Then $H$ contains a $d$-flow (or equivalently, a $d$-linkage) if and only if $\gain{d}(C)\ge 0$ for every cycle $C$ in $\dual{H}$.
\end{observation}
If $\gain{d}(C)\ge 0$ for every cycle $C$ in $\dual{H}$, then we say that the function $d$ is \emph{feasible}.
For a set $X$ of cycles and edges of $\dual{H}$, we let $\gain{d}(X)=\sum_{x\in X} \gain{d}(x)$.
  For an edge $e\in E(\dual{H})$ contained in $a$ cycles and $b$ edges from $X$, we define $m(X,e)=a+2b$.
If $s$ and $t$ are faces of $H$, we say that $X$ is \emph{$(s,t)$-connecting}
if $\bigcup X$ contains a path from $\dual{s}$ to $\dual{t}$.  We say that $X$ is \emph{laminar}
if for any cycles $C_1,C_2\in X$, the open disks in the plane bounded by $C_1$ and $C_2$
are either disjoint or one is a subset of the other.

If $d$ is feasible, then let $\circul(H,s,t,d)$ denote the maximum integer $k$ such that 
there exists a $d$-linkage and an $(s,t)$-circulation of size $k$ in $H$ with disjoint supports.  We are now ready to prove the main
result of this section, the min-max theorem for $\circul(H,s,t,d)$.  The argument used to prove the part (a) is based on the idea
of Seymour~\cite{twocom} for $2$-commodity flows.

\begin{theorem}\label{thm-flowchar}
Let $H$ be a connected plane graph, let $s$ and $t$ be faces of $H$, and let $d:V(H)\to \mathbf{Z}$ be a feasible even function such that
$d(V(H))=0$.
Let $h_1$ and $h_2$ be $(s,t)$-circulation-maximum and minimum $d$-flows in $H$, respectively.
Then the following claims hold.
\begin{itemize}
\item[(a)] For every path $Q$ from $\dual{s}$ to $\dual{t}$ in $\dual{H}$, $\circul(H,s,t,d)=\frac{1}{2}\Bigl(\int_Q \dual{h_1}-\int_Q \dual{h_2}\Bigr)$.
\item[(b)] $\gain{d}(X)\ge 2\cdot\circul(H,s,t,d)$ for every $(s,t)$-connecting set $X$ of cycles and edges of $\dual{H}$.
\item[(c)] There exists a laminar $(s,t)$-connecting set $X$ of cycles and edges of $\dual{H}$
such that $\gain{d}(X)=2\cdot\circul(H,s,t,d)$.
\end{itemize}
Furthermore, there exists an algorithm with time complexity $O(|d|\cdot\Vert H\Vert)$ which given $H$, $s$, $t$, and $d$ returns
\begin{itemize}
\item a $d$-linkage and an $(s,t)$-circulation of size $\circul(H,s,t,d)$ in $H$ with disjoint supports, and
\item a laminar $(s,t)$-connecting set $X$ of cycles and edges of $\dual{H}$ such that $\gain{d}(X)=2\cdot\circul(H,s,t,d)$.
\end{itemize}
\end{theorem}
\begin{proof}
Suppose that $\PP$ is a $d$-linkage and $\CC$ is an $(s,t)$-circulation in $H$ such that $\supp(\PP)\cap\supp(\CC)=\emptyset$
and $|\CC|$ is maximum, i.e., equal to $\circul(H,s,t,d)$.
For any cycle $K$ in $\dual{H}$, the set $\{e:\dual{e}\in E(K)\}$ forms an edge-cut in $H$ separating $\inter(K)$ from $\ext(K)$.
Clearly, $\PP$ contains at least $d(K)$ paths with one end in $\inter(K)$ and the other hand in $\ext(K)$, and thus
all but at most $\gain{d}(K)$ edges of $K$ are intersected by paths in $\PP$.

Consider now an $(s,t)$-connecting set $X$ of cycles and edges of $\dual{H}$.  Let $F$ be the multigraph obtained from
$\bigcup X$ by giving each edge $e$ the multiplicity $m(X, e)$ (i.e., when constructing $F$, the cycles of $X$ are made edge-disjoint and edges of $X$ are turned into $2$-cycles
by increasing the multiplicity of edges).  According to the previous paragraph, all but at most $\sum_{x\in X} \gain{d}(x)=\gain{d}(X)$ edges of $F$ are intersected by paths in $\PP$.
All vertices of $F$ have even degree, and thus each component of $F$ is $2$-edge-connected.  Since $X$ is $(s,t)$-connecting, it follows that $F$ contains
two edge-disjoint paths from $\dual{s}$ to $\dual{t}$, and thus every $(s,t)$-non-contractible cycle in $H$ intersects at least two edges of $F$.
Since the supports of $\PP$ and $\CC$ are edge-disjoint, we conclude that $\gain{d}(X)\ge 2|\CC|=2\cdot\circul(H,s,t,d)$.  Therefore, (b) holds.

Let $Q$ be any path from $\dual{s}$ to $\dual{t}$ in $\dual{H}$.
By Observation~\ref{obs-flicir}, there exists a $d$-flow $h_d$ and a $\varnothing$-flow $h_c$ in $H$ with $\supp(h_d)=\supp(\PP)$,
$\supp(h_c)=\supp(\CC)$, and $\int_Q \dual{h_c}=|\CC|$.  In particular, the supports of $h_d$ and $h_c$ are disjoint, and thus
$h_a=h_d+h_c$ and $h_b=h_d-h_c$ are $d$-flows in $H$.  Since $h_1$ and $h_2$ are $(s,t)$-circulation-maximum and minimum, respectively,
we have $\int_Q \dual{h_a}\le \int_Q \dual{h_1}$ and $\int_Q \dual{h_b}\ge \int_Q \dual{h_2}$.  Consequently,
\begin{align*}
\circul(H,s,t,d)&=|\CC|=\int_Q \dual{h_c}=\int_Q \frac{1}{2}(\dual{h_a}-\dual{h_b})\\
&=\frac{1}{2}\Biggl(\int_Q \dual{h_a}-\int_Q \dual{h_b}\Biggr)\le \frac{1}{2}\Biggl(\int_Q \dual{h_1}-\int_Q \dual{h_2}\Biggr).
\end{align*}
On the other hand, by Lemma~\ref{lemma-circmax}, since $d$ is even we can assume that $\supp(h_1)=\supp(h_2)=E(H)$.
Let $h_+=(h_1+h_2)/2$ and $h_-=(h_1-h_2)/2$.
Since $\supp(h_1)=\supp(h_2)=E(H)$, all values of $h_1$ and $h_2$ are odd, and thus all values of $h_+$ and $h_-$ are integers.
Consequently, $h_+$ is a $d$-flow and $h_-$ is a $\varnothing$-flow in $H$.  Furthermore, observe that $\supp(h_+)\cap \supp(h_-)=\emptyset$.
By Observation~\ref{obs-flicir}, we conclude that $H$ contains a $d$-linkage $\PP'$ and an $(s,t)$-circulation $\CC'$
with disjoint supports such that $|\CC'|=\int_Q \dual{h_-}=\frac{1}{2}\Bigl(\int_Q \dual{h_1}-\int_Q \dual{h_2}\Bigr)$.  Therefore, $\circul(H,s,t,d)\ge \frac{1}{2}\Bigl(\int_Q \dual{h_1}-\int_Q \dual{h_2}\Bigr)$.
Combining the inequalities, we conclude that (a) holds.  Note that $(s,t)$-circulation-maximum and minimum $d$-flows
with maximal supports can be found in time $O(|d|\cdot\Vert H\Vert)$ using the algorithm of Lemma~\ref{lemma-circmax-alg},
and they can be converted into a $d$-linkage and an $(s,t)$-circulation of size $\circul(H,s,t,d)$ in $H$ with disjoint supports
in time $O(\Vert H\Vert)$ as described in this paragraph.

Finally, let us prove the part (c).
By Lemma~\ref{lemma-circmax}, there exist paths $Q_1$ and $Q_2$ in $\dual{H}$ from $\dual{s}$ to $\dual{t}$ such that
$\int_{Q_1} h_1=|E(Q_1)|$ and $\int_{Q_2} h_2=-|E(Q_2)|$.
Let $R$ be the closed walk obtained as the concatenation of $Q_1$ with the reversal of $Q_2$.
By Lemma~\ref{lemma-flowover}, we have $\int_{Q_1} \dual{h_2}=\int_{Q_2} \dual{h_2}+\sum_{f\in F(\dual{H})} \omega_R(f)\dual{d}(f)$,
and thus
\begin{align}
2\cdot\circul(H,s,t,d)&=\int_{Q_1} \dual{h_1}-\int_{Q_1} \dual{h_2}\nonumber\\
&=\int_{Q_1} \dual{h_1}-\int_{Q_2} \dual{h_2}-\sum_{f\in F(\dual{H})} \omega_R(f)\dual{d}(f)\nonumber\\
&=|E(Q_1)|+|E(Q_2)|-\sum_{f\in F(\dual{H})} \omega_R(f)\dual{d}(f)\label{eq-ctos}.
\end{align}
For an integer $i$, let $L_i$ be the set consisting of edges $e$ of $\dual{H}$ such that $e$ is incident with faces $f_1$ and $f_2$ satisfying
$\omega_R(f_1)\ge i$ and $\omega_R(f_2)<i$ (one can view $\omega_R$ as assigning heights to faces of $\dual{H}$, and $L_i$
then corresponds to the contour lines at height very slightly less than $i$).
Let $A$ be the set of edges of $\dual{H}$ through that both $Q_1$ and $Q_2$ pass in the same direction.
We claim that $X$ can be chosen to consist of the edges of $A$ and of the cycles into which the sets $L_i$ naturally decompose.
Let us describe the construction precisely.

By Observation~\ref{obs-wind}(b), for an edge $e$ of $\dual{H}$ incident with faces $f_1$ and $f_2$ and letting
$n=\max(\omega_R(f_1),\omega_R(f_2))$,
\begin{itemize}
\item if $e\in (E(Q_1)\setminus E(Q_2))\cup (E(Q_2)\setminus E(Q_1))$, then $e$ belongs to exactly one of the sets $L_i$,
namely to $L_n$,
\item if $e\in E(Q_1)\cap E(Q_2)\setminus A$, then $e$ belongs exactly to two of the sets, namely to $L_n$ and $L_{n-1}$, and
\item if $e\in A$ or $e\not\in E(Q_1)\cup E(Q_2)$, then $e$ does not belong to any of the sets.
\end{itemize}
It follows that
\begin{equation}\label{eq-sqs}
|E(Q_1)|+|E(Q_2)|=2|A|+\sum_i |L_i|.
\end{equation}

Let $\overline{L_i}$ denote the subgraph of $\dual{H}$ with the edge set $L_i$ and the vertex set consisting of the
vertices incident with the edges of $L_i$.
Observe that if $f$ is a face of $\overline{L_i}$ and faces $f_1$ and $f_2$ of $\dual{H}$ satisfy $f_1,f_2\subseteq f$, then
either both $\omega_R(f_1)\ge i$ and $\omega_R(f_2)\ge i$, or both $\omega_R(f_1)<i$ and $\omega_R(f_2)<i$.
Let $F^+_i$ and $F^-_i$ denote the sets of faces of $\overline{L_i}$ for that the former or the latter, respectively, holds.
For $f\in F^+_i$, let $W_f$ denote the subgraph of $\overline{L_i}$ drawn in the boundary of $f$.
Note that faces in $F^+_i$ only share edges with faces in $F^-_i$ and vice-versa.  Consequently, $\overline{L_i}=\bigcup_{f\in F^+_i} W_f$,
and the graphs $W_f$ for $f\in F^+_i$ are pairwise edge-disjoint and $2$-edge-connected.  For $f\in F^+_i$, let $K_f$
denote the set of $2$-connected blocks of $W_f$; since $W_f$ is $2$-edge-connected and all its edges are incident with $f$,
$K_f$ is a set of cycles.  For $C\in K_f$, let $\out_f(C)=\dual{\ext}(C)$ if $f$ is contained in the open disk of the plane bounded by $C$,
and let $\out_f(C)=\dual{\inter}(C)$ otherwise.  Since $d(V(H))=0$, for each face $f\in F^+_i$ we have
\begin{equation}\label{eq-dkf}
\sum_{f'\in F(\dual{H}), f'\subseteq f} \dual{d}(f')=-\sum_{f'\in F(\dual{H}), f'\not\subseteq f} \dual{d}(f')=-\sum_{C\in K_f}\sum_{f''\in \out_f(C)} \dual{d}(f'')\le \sum_{C\in K_f} d(C).
\end{equation}
Let $K_i=\bigcup_{f\in F^+_i} K_f$ and $X=A\cup \bigcup_i K_i$.  Clearly, $X$ is $(s,t)$-connecting, since $\bigcup X=Q_1\cup Q_2$.
Observe that if $f\in F_i^+$ and $i>j$, then there exists a face $f'\in F_j^+$ such that
$f\subseteq f'$, and thus the set $X$ is laminar.  It remains to argue that $\gain{d}(X)=2\cdot\circul(H,s,t,d)$.

Let $m$ be the minimum of $\{\omega_R(f):f\in F(\dual{H})\}$.  Note that for any $f'\in F(\dual{H})$,
\begin{equation}\label{eq-steps}
\omega_R(f')=m+|\{i>m:\text{$f'\subseteq f$ for some $f\in F^+_i$}\}|.
\end{equation}
By (\ref{eq-steps}) and (\ref{eq-dkf}), and using the fact that $d(V(H))=0$, we have
$$\sum_{f'\in F(\dual{H})} \omega_R(f')\dual{d}(f')=m\cdot \dual{d}(F(\dual{H}))+\sum_{i>m}\sum_{f\in F^+_i}\sum_{f'\in F(\dual{H}), f'\subseteq f} d(f')\le \sum_{i>m} \sum_{C\in K_i} d(C),$$
and thus by (\ref{eq-ctos}),
\begin{equation}\label{eq-ctosl}
2\cdot\circul(H,s,t,d)\ge|E(Q_1)|+|E(Q_2)|-\sum_{i>m} \sum_{C\in K_i} d(C).
\end{equation}
By (\ref{eq-sqs}), we have
$$\sum_{x\in X} |x|=2|A|+\sum_i |L_i|=|E(Q_1)|+|E(Q_2)|,$$
and the definition of $X$ gives
$$\sum_{x\in X} d(x)=\sum_{i>m} \sum_{C\in K_i} d(C).$$
By (\ref{eq-ctosl}), we conclude that $$2\cdot\circul(H,s,t,d)\ge |E(Q_1)|+|E(Q_2)|-\sum_{i>m} \sum_{C\in K_i} d(C)=\sum_{x\in X} |x|-\sum_{x\in X} d(x)=\gain{d}(X).$$   By (b), we conclude that $\gain{d}(X)=2\cdot\circul(H,s,t,d)$,
and thus (c) holds.  Furthermore, observe that this construction of the set $X$ from the $d$-flows $h_1$ and $h_2$ can be performed
in time $O(\Vert H\Vert)$.
\end{proof}

For the $3$-coloring applications, we need the following consequence.

\begin{corollary}\label{alg-minmax}
There exists an algorithm with time complexity $O(|d|\cdot\Vert H\Vert)$ which given 
a connected plane graph $H$, faces $s$ and $t$ of $H$, a path $Q$ from $\dual{s}$ to $\dual{t}$ in $\dual{H}$,
an even function $d:V(H)\to \mathbf{Z}$ such that $d(V(H))=0$, and an integer $m$ returns one of the following:
\begin{itemize}
\item A $d$-flow $h$ in $H$ such that $\supp(h)=E(H)$ and $\int_Q \dual{h}\equiv m\pmod 3$, or
\item a cycle $C$ in $\dual{H}$ such that $\gain{d}(C)<0$, or
\item a laminar $(s,t)$-connecting set $X$ of cycles and edges of $\dual{H}$ with $\gain{d}(X)\le 2$,
and $d$-flows $h_1$ and $h_2$ in $H$ such that $\int_Q \dual{h_2} = \int_Q \dual{h_1}+\gain{d}(X)$ and
$\int_Q \dual{h_1}\not\equiv m\not\equiv \int_Q \dual{h_2}\pmod 3$.
\end{itemize}
\end{corollary}
\begin{proof}
If $d$ is not feasible, then a straightforward modification of Ford-Fulkerson algorithm returns a cycle $C$ in $\dual{H}$
such that $\gain{d}(C)<0$.  We can return such a cycle and stop.

Hence, suppose that $d$ is feasible.  Let
$\PP$ be a $d$-linkage and let $\CC$ be an $(s,t)$-circulation of size $\circul(H,s,t,d)$ in $H$ with $\supp(\PP)\cap\supp(\CC)=\emptyset$,
and let $X$ be a laminar $(s,t)$-connecting set $X$ of cycles and edges of $\dual{H}$ such that $\gain{d}(X)=2\cdot\circul(H,s,t,d)$
returned by the algorithm from Theorem~\ref{thm-flowchar}.  Let $h_0$ be a $d$-flow with $\supp(h_0)=\supp(\PP)$ obtained from $\PP$
by Observation~\ref{obs-flicir}.  Let $H'$ be the subgraph of $H$ with vertex set $V(H)$ and edge set $E(H)\setminus (\supp(\PP)\cup \supp(\CC))$.
Since $h_0$ is even, all vertices have even degree in $H'$, and by the maximality of $\CC$, the graph $H'$ does not contain
any $(s,t)$-non-contractible cycle.  Hence, we can express $H'$ as an edge-disjoint union of $(s,t)$-contractible cycles,
and by sending one unit of flow along each of them, we obtain a $\varnothing$-flow $h'$ in $H$ such that
$\supp(h_0)\cup \supp(h')\cup \supp(\CC)=E(H)$, the supports are pairwise disjoint, and $\int_Q \dual{h'}=0$.
Let $n=\min(\circul(H,s,t,d)+1,3)$, and for $i\in\{1,\ldots,n\}$, let $c_i$ be a $\varnothing$-flow with $\supp(c_i)=\supp(\CC)$
and $\int_Q \dual{c_i}=-|\CC|+2i-2$, obtained by Observation~\ref{obs-flicir}.  Let $h_i=h_0+h'+c_i$, and note that $h_i$
is a $d$-flow with $\supp(h_i)=E(H)$ and $\int_Q \dual{h_i}=\int_Q \dual{h_1} + 2i-2$.  If there exists $i\in \{1,\ldots, n\}$ such that
$\int_Q \dual{h_i}\equiv m\pmod 3$, we return $h=h_i$ and stop.

Otherwise, we clearly have $n\le 2$, and thus $\circul(H,s,t,d)\le 1$ and $\gain{d}(X)\in\{0,2\}$.
If $\gain{d}(X)=2$, then return $X$, $h_1$, and $h_2$.  If $\gain{d}(X)=0$, then return $X$, $h_1$, and $h_1$.
\end{proof}

Note that the last two outcomes of Corollary~\ref{alg-minmax} certify that the first one is impossible:
the second one because no $d$-flow exists, the third one because $\int_Q \dual{h}\in\Bigl\{\int_Q \dual{h_1},\int_Q \dual{h_2}\Bigr\}$ for every $d$-flow $h$ in $H$
with $\supp(h)=E(H)$.

\section{Coloring graphs in the cylinder}\label{sec-cyl}

Let $G$ be a connected plane graph and let $\vec{G}$ be an orientation of $G$.
Given a $3$-coloring $\varphi:V(G)\to \{0,1,2\}$ of $G$, we define a function
$\delta_{\vec{G},\varphi}:E\bigl(\vec{G}\bigr)\to \{-1,1\}$ so that for every edge $e=(u,v)\in E\bigl(\vec{G}\bigr)$, we have
$\varphi(v)\equiv \varphi(u)+\delta_{\vec{G},\varphi}(e)\pmod 3$.
Note that if $Q$ is a path from a vertex $s$ to a vertex $t$ in $G$, this implies
$\varphi(t)\equiv \varphi(s)+\int_Q \delta_{\vec{G},\varphi}\pmod 3$.
For a plane graph $H$ and a function $d:V(H)\to \mathbf{Z}$, if $d(v)$ is divisible by $3$ for every $v\in V(H)$, we write $3|d$.

Tutte~\cite{tutteflow} observed a connection between colorings of a plane graph and flows in its dual, which we can
restate in the precoloring extension setting as follows.  

\begin{lemma}\label{tutte-flow}
Let $G$ be a connected plane graph, let $H=\dual{G}$, let $\vec{H}$ be an arbitrary orientation of $H$ and let $\vec{G}=\dual{\vec{H}}$.
Let $C_1$ and $C_2$ be connected subgraphs of $G$, let $Q$ be a path in $G$ from a vertex
$v_1\in V(C_1)$ to a vertex $v_2\in V(C_2)$, 
let $C=C_1\cup C_2$ and let $\vec{C}$ be the orientation of $C$ induced by $\vec{G}$.
A proper $3$-coloring $\psi:V(C)\to\{0,1,2\}$ of $C$
extends to a $3$-coloring of $G$ if and only if there exists a
feasible even function $d:V(H)\to \mathbf{Z}$ such that $d(V(H))=0$ and $3|d$,
and a $d$-flow $h$ in $H$ with $\supp(h)=E(H)$ such that the restriction of $\dual{h}$ to $E\bigl(\vec{C}\bigr)$ is equal to
$\delta_{\vec{C},\psi}$
and $\psi(v_2)\equiv \psi(v_1)+\int_Q \dual{h}\pmod 3$.  Given such a $d$-flow $h$, we can obtain a $3$-coloring of $G$
extending $\psi$ in time $O(\Vert G\Vert)$.
\end{lemma}
\begin{proof}
Consider a $3$-coloring $\varphi:V(G)\to\{0,1,2\}$ of $G$ extending $\psi$, and let us define $h(e)=\delta_{\vec{G},\varphi}(\dual{e})$
for every $e\in E\bigl(\vec{H}\bigr)$, so that $\delta_{\vec{G},\varphi}=\dual{h}$.  For every $v\in V(H)$, define
$d(v)=\sum_{e\in N_{\vec{H}}^+(v)} h(e)-\sum_{e\in N_{\vec{H}}^-(v)} h(e)$,
so that $h$ is a $d$-flow in $H$; clearly, $d$ is feasible.  Since $h(e)\in \{-1,1\}$
for all $e\in E\bigl(\vec{H}\bigr)$, we conclude that $d(v)$ and $\deg_H(v)$ have the same parity for all $v\in V(H)$, and
thus $d$ is even.  Let $W=u_1u_2\ldots u_t$ be the facial walk of the face $\dual{v}$ of $G$,
directed so that $\dual{v}$ is to the right from it, and let $u_{t+1}=u_1$.  Then by the definition of $d$ and $h$, we have
\begin{align*}
d(v)&=\sum_{e\in N_{\vec{H}}^+(v)} h(e)-\sum_{e\in N_{\vec{H}}^-(v)} h(e)=\int_W \dual{h}\\
&=\int_W \delta_{\vec{G},\varphi} \equiv \sum_{i=1}^t \varphi(u_{i+1})-\varphi(u_i)=0\pmod 3,
\end{align*}
and thus $3|d$.  Furthermore, if $\varphi$ extends $\psi$, then $\delta_{\vec{C},\psi}$ is equal to the restriction
of $\delta_{\vec{G},\varphi}=\dual{h}$ to $E\bigl(\vec{C}\bigr)$, and 
$\varphi(v_2)\equiv \varphi(v_1)+\int_Q \delta_{\vec{G},\varphi}=\varphi(v_1)+\int_Q \dual{h}\pmod 3$.
Hence, the conclusions of the lemma hold.

Suppose conversely that $d:V(H)\to \mathbf{Z}$ is a feasible even function such that $d(V(H))=0$ and $3|d$,
and let $h$ be a $d$-flow in $H$ with $\supp(h)=E(H)$ such that the restriction of $\dual{h}$ to $E\bigl(\vec{C}\bigr)$ is
equal to $\delta_{\vec{C},\psi}$ and $\psi(v_2)\equiv \psi(v_1)+\int_Q \dual{h}\pmod 3$.
Since $h$ is a $d$-flow and $3|d$, by Lemma~\ref{lemma-flowover}, we have
$\int_R \dual{h}=\nabla(R,\dual{d})\equiv 0\pmod 3$ for every closed walk $R$ in $G$.
For each vertex $x\in V(G)$, let $Q_x$ be any path from $v_1$ to $x$ in $G$.
Let $\theta:V(G)\to \{0,1,2\}$ be the function satisfying $\theta(x)\equiv\psi(v_1)+\int_{Q_x} \dual{h}\pmod 3$
for every $x\in V(G)$.  For any walk $W$ from a vertex $x$ to a vertex $y$ of $G$, let $R_W$ denote the concatenation
of $Q_x$, $W$, and the reversal of $Q_y$.  Since $R_W$ is a closed walk, we have
\begin{equation}\label{eq:walk}
\int_W \dual{h}=\int_{R_W}\dual{h}+\int_{Q_y}\dual{h}-\int_{Q_x}\dual{h}\equiv\int_{Q_y}\dual{h}-\int_{Q_x}\dual{h}\equiv \theta(y)-\theta(x)\pmod 3.
\end{equation}
Since $\supp(h)=E(H)$, if $e=(x,y)$ is an edge of $\vec{G}$, then letting $Q_e$ be the path consisting only of $e$,
we have
$$\theta(y)-\theta(x)\equiv\int_{Q_e} \dual{h}=\dual{h}(e)\in \{-1,1\}\pmod 3.$$
Consequently, $\theta(x)\neq\theta(y)$, and thus $\theta$ is a proper $3$-coloring of $G$.
Clearly $\theta(v_1)=\psi(v_1)$, and (\ref{eq:walk}) implies
$$\theta(v_2)\equiv \theta(v_1)+\int_{Q_{v_2}} \dual{h}=\psi(v_1)+\int_{Q_{v_2}} \dual{h}\equiv \psi(v_2)\pmod 3$$
by the assumptions.  Since $\theta(v_2),\psi(v_2)\in \{0,1,2\}$, we have $\theta(v_2)=\psi(v_2)$.
For $i\in\{1,2\}$ and any vertex $x\in V(C_i)$, let $P_x$ denote a path in $C_i$ from $v_i$ to $x$.
By (\ref{eq:walk}) and the assumptions, we have
\begin{align*}
\theta(x)&\equiv \theta(v_i)+\int_{P_x} \dual{h}=\psi(v_i)+\int_{P_x} \dual{h}\\
&=\psi(v_i)+\int_{P_x} \delta_{\vec{C},\psi}\equiv \psi(v_i)+(\psi(x)-\psi(v_i))\\
&=\psi(x)\pmod 3,
\end{align*}
and since $\theta(x),\psi(x)\in \{0,1,2\}$, we have $\theta(x)=\psi(x)$.  Consequently, $\theta$ is a proper $3$-coloring of $G$
extending $\psi$.  Note that $\theta$ can be constructed from $h$ in time $O(\Vert G\Vert)$, setting $\theta(v_1)=\psi(v_1)$
and propagating the colors using the fact that $\theta(y)\equiv \theta(x)+\dual{h}(e)\pmod 3$ for every edge $e=(x,y)$ of $\vec{G}$.
\end{proof}

When applying Lemma~\ref{tutte-flow}, we usually try all feasible even functions $d:V(H)\to \mathbf{Z}$ such that $d(V(H))=0$ and $3|d$
one by one.  Once the function $d$ is fixed, we can enforce the condition that the restriction of $\dual{h}$ to $E\bigl(\vec{C}\bigr)$ is equal to
$\delta_{\vec{C},\psi}$ using the following construction.

Let $G$ be a connected plane graph, let $H=\dual{G}$, let $\vec{H}$ be an arbitrary orientation of $H$, and let $\vec{G}=\dual{\vec{H}}$.
Let $z_1$ and $z_2$ be vertices of $H$ and let $C$ be the subgraph of $G$ consisting of the vertices and edges drawn in the boundaries
of the faces $\dual{z_1}$ and $\dual{z_2}$.  Let $\vec{C}$ be the orientation of $C$ induced by $\vec{G}$.
For functions $d:V(H)\to \mathbf{Z}$ and $\delta:E\bigl(\vec{C}\bigr)\to\mathbf{Z}$,
let $(d/\delta):V(H-\{z_1,z_2\})\to\mathbf{Z}$ be the function defined by setting
$$(d/\delta)(v)=d(v)+\sum_{i\in\{1,2\}}\Biggl(\sum_{e\in N^+_{\vec{H}}(z_i)\cap N^-_{\vec{H}}(v)} \delta(\dual{e})-\sum_{e\in N^-_{\vec{H}}(z_i)\cap N^+_{\vec{H}}(v)} \delta(\dual{e})\Biggr)$$
for every $v\in V(H)\setminus\{z_1,z_2\}$.  That is, 
$(d/\delta)(v)$ is the amount of charge originating in $v$ adjusted by
adding the amount of charge sent to $v$ from $z_1$ and $z_2$ according to the partial flow corresponding to~$\delta$.

\begin{observation}\label{obs-contrflow}
Let $G$ be a connected plane graph, let $H=\dual{G}$, let $\vec{H}$ be an arbitrary orientation of $H$, and let $\vec{G}=\dual{\vec{H}}$.
Let $z_1$ and $z_2$ be vertices of $H$ and let $C$ be the subgraph of $G$ consisting of the vertices and edges drawn in the boundaries
of the faces $\dual{z_1}$ and $\dual{z_2}$.  Let $\vec{C}$ be the orientation of $C$ induced by $\vec{G}$.
Let $d:V(H)\to \mathbf{Z}$ and $\delta:E\bigl(\vec{C}\bigr)\to\mathbf{Z}$ be functions such that
$$d(z_i)=\sum_{e\in N^+_{\vec{H}}(z_i)} \delta(\dual{e})-\sum_{e\in N^-_{\vec{H}}(z_i)} \delta(\dual{e})$$
for $i\in\{1,2\}$.
\begin{itemize}
\item If $h'$ is a $(d/\delta)$-flow in $H-\{z_1,z_2\}$, then the function $h:E\bigl(\vec{H}\bigr)\to\mathbf{Z}$
defined by $h(e)=h'(e)$ for $e\in E\bigl(\vec{H}-\{z_1,z_2\}\bigr)$ and $h(e)=\delta(\dual{e})$ for each edge $e$
incident with $z_1$ or $z_2$ is a $d$-flow in $H$.
\item If $h$ is a $d$-flow in $H$ such that $h(e)=\delta(\dual{e})$ for each edge $e$ incident with $z_1$ or $z_2$,
then the restriction of $h$ to $E\bigl(\vec{H}-\{z_1,z_2\}\bigr)$ is a $(d/\delta)$-flow in $H-\{z_1,z_2\}$.
\end{itemize}
\end{observation}

We can now put things together to obtain an algorithm to $3$-color graphs drawn in the cylinder with precolored boundary cycles.
\begin{proof}[Proof of Theorem~\ref{thm-mainalg}]
If the cycles $C_1$ and $C_2$ are not disjoint, then it suffices to check whether $\psi$ extends to a $3$-coloring
of each subgraph of $G$ drawn in the closure of a face of the connected graph $C_1\cup C_2$.
This can be done using the algorithm from Theorem~\ref{thm-onec}, applied separately to each such subgraph.
Hence, we can assume that the cycles $C_1$ and $C_2$ are vertex-disjoint.

Let $H=\dual{G}$, let $\vec{H}$ be an arbitrary orientation of $H$, and let $\vec{G}=\dual{\vec{H}}$.
Let $C=C_1\cup C_2$ and let $\vec{C}$ be the orientation of $C$ induced by $\vec{G}$.  Let $z_1$ and $z_2$ be the vertices of $H$
such that $\dual{z_i}$ is the face of $G$ bounded by $C_i$ for $i\in\{1,2\}$.  We view the cycles $C_1$ and $C_2$ as directed so
that the face $\dual{z_i}$ is to the right from $C_i$.
Let $Q$ be any path in $G$ from a vertex $v_1\in V(C_1)$ to a vertex $v_2\in V(C_2)$ intersecting $C_1\cup C_2$ only in its endvertices,
and let $m=(\psi(v_2)-\psi(v_1))\bmod 3$.
We iterate over all even functions $d:V(H)\to \mathbf{Z}$ such that $d(V(H))=0$, $3|d$,
$d(z_i)=\int_{C_i} \delta_{\vec{C},\psi}$ for $i\in \{1,2\}$, and $|d(v)|\le \deg_H(v)$ for all $v\in V(H)$.
Clearly, there are at most $r_{f_1,f_2}(G)$ possible choices for $d$.  By Lemma~\ref{tutte-flow}, it suffices to check whether
for any such function $d$,
there exists a $d$-flow $h$ in $H$ with $\supp(h)=E(H)$ such that the restriction of
$\dual{h}$ to $E\bigl(\vec{C}\bigr)$ is equal to $\delta_{\vec{C},\psi}$ and $\int_Q \dual{h}\equiv m\pmod 3$.
If such a $d$-flow exists, we can in time $O(|G|)$ turn $h$ into a $3$-coloring extending $\psi$
as described in the proof of Lemma~\ref{tutte-flow}.  Otherwise, Lemma~\ref{tutte-flow} implies no $3$-coloring of $G$
extends $\psi$.

Let $g_1$ and $g_2$ be the faces of $H-\{z_1,z_2\}$ such that in $H$, the vertex $z_i$ is drawn in $g_i$ for $i\in\{1,2\}$
(since the cycles $C_1$ and $C_2$ are vertex-disjoint, we have $g_1\neq g_2$).  Note that the graph $G'=\dual{(H-\{z_1,z_2\})}$
contains a path $Q'$ from $\dual{g_1}$ to $\dual{g_2}$ with $E(Q')=E(Q)$.  By Observation~\ref{obs-contrflow}, to find
a $d$-flow $h$ with the properties described in the previous paragraph (or decide none exists), it suffices to find
a $(d/\delta_{\vec{C},\psi})$-flow $h'$ in $H-\{z_1,z_2\}$ such that $\supp(h')=E(H-\{z_1,z_2\})$
and $\int_{Q} \dual{h'}\equiv m\pmod 3$, or decide none exists.  By Corollary~\ref{alg-minmax}, this can be done in
time $O(|d|\cdot|G|)=O(q(G)|G|)$.

Since the test needs to be performed for at most $r_{f_1,f_2}(G)$ possible choices of $d$, we conclude the time complexity
of the algorithm is $O(r_{f_1,f_2}(G)q(G)|G|)$.
\end{proof}

Using the ideas from the proof of Theorem~\ref{thm-mainalg}, let us now explicitly formulate a sufficient condition
for extendability of a $3$-coloring of two cycles in a planar graph.  Let $G$ be a connected simple plane graph and let $C_1$ and $C_2$
be vertex-disjoint cycles bounding distinct faces $f_1$ and $f_2$ of $G$, where $f_2$ is the outer face of $G$; for $i\in\{1,2\}$, we view the cycle $C_i$ as directed so
that the face $f_i$ is to the right from $C_i$.  Let $\vec{C}$ be any orientation of $C_1\cup C_2$.
Let $\psi$ be a proper $3$-coloring of $C_1\cup C_2$
and let $\dual{d}:F(G)\to \mathbb{Z}$ be an even function such that $\dual{d}(F(G))=0$, $3|\dual{d}$,
$\dual{d}(f_i)=\int_{C_i} \delta_{\vec{C},\psi}$ for $i\in \{1,2\}$, and $|\dual{d}(f)|\le |f|$ for all $f\in F(G)$.

If a path $R$ in $G$ has both ends in $C_i$ for some $i\in\{1,2\}$
and is otherwise disjoint from $C_i$ and edge-disjoint from $C_{3-i}$, we say $R$ is a \emph{generalized chord} of $C_i$.  Let $K_R$ be the unique
$(f_1,f_2)$-contractible cycle in $C_i\cup R$, and let $B$ be the path $K_R\cap C_i$ directed so that $f_i$ is to the right of $B$
(we say $B$ is the \emph{base} of $R$).  We define $\dual{\inter}(R)=\dual{\inter}(K_R)$.
Let us remark that since $f_2$ is the outer face of $G$, we have $f_1,f_2\not\in \dual{\inter}(R)$.
We define $$\gain{\dual{d},\psi}(R)=|E(R)|-\Bigl|\dual{d}(\dual{\inter}(R))+\int_B \delta_{\vec{C},\psi}\Bigr|.$$

A \emph{$(C_1,C_2)$-connector} $Q$ is the union of two vertex-disjoint paths, both with one end in $C_1$, the other end in $C_2$,
and otherwise disjoint from $C_1\cup C_2$.  Let $K_Q$ be one of the two $(f_1,f_2)$-contractible cycles in $C_1\cup C_2\cup Q$,
and for $i\in\{1,2\}$, let $B_i$ be the path $K_Q\cap C_i$ directed so that $f_i$ is to the right of $B_i$.
We define $$\gain{\dual{d},\psi}(Q)=|E(Q)|-\Bigl|\dual{d}(\dual{\inter}(K_Q))+\int_{B_1} \delta_{\vec{C},\psi}+\int_{B_2} \delta_{\vec{C},\psi}\Bigr|.$$
Observe that the value of $\gain{\dual{d},\psi}(Q)$ does not depend on which of the two cycles we choose as $K_Q$,
since $\dual{d}(F(G))=0$ and $\dual{d}(f_i)=\int_{C_i} \delta_{\vec{C},\psi}$ for $i\in\{1,2\}$.

For a cycle $K$ in $G$ edge-disjoint from $C_1\cup C_2$, we let $\gain{\dual{d},\psi}(K)=|K|-|\dual{d}(\dual{\inter}(K))|$.
An edge $e$ of $G$ is \emph{non-chord} if it is not the case that both vertices incident with $e$ are contained in the same cycle $C_i$,
for $i\in\{1,2\}$.  For a non-chord edge $e$, we let $\gain{\dual{d},\psi}(e)=2$.

A \emph{constraint} is a generalized chord, a $(C_1,C_2)$-connector, a cycle edge-disjoint from $C_1\cup C_2$, or a non-chord edge.
A set $X$ of constraints is \emph{$(C_1,C_2)$-connecting} if $\bigcup X$ contains a path from $C_1$ to $C_2$.
We define $\gain{\dual{d},\psi}(X)=\sum_{T\in X} \gain{\dual{d},\psi}(T)$.

\begin{lemma}\label{lemma-sufficient}
Let $G$ be a connected simple plane graph and let $C_1$ and $C_2$
be vertex-disjoint cycles bounding distinct faces $f_1$ and $f_2$ of $G$, where $f_2$ is the outer face of $G$; for $i\in\{1,2\}$, we view the cycle $C_i$ as directed so
that the face $f_i$ is to the right from $C_i$.  Let $\vec{C}$ be any orientation of $C_1\cup C_2$.
Let $\psi$ be a proper $3$-coloring of $C_1\cup C_2$
and let $\dual{d}:F(G)\to \mathbb{Z}$ be an even function such that $\dual{d}(F(G))=0$, $3|\dual{d}$,
$\dual{d}(f_i)=\int_{C_i} \delta_{\vec{C},\psi}$ for $i\in \{1,2\}$, and $|\dual{d}(f)|\le |f|$ for all $f\in F(G)$.
If $\gain{\dual{d},\psi}(T)\ge 0$ for every constraint $T$ and $\gain{\dual{d},\psi}(X)>2$ for every $(C_1,C_2)$-connecting
set $X$ of constraints, then $\psi$ extends to a $3$-coloring of $G$.
\end{lemma}
\begin{proof}
Let $H=\dual{G}$, let $\vec{H}$ be an orientation of $H$, and let $\vec{G}=\dual{\vec{H}}$; we choose $\vec{H}$ so that the orientation
$\vec{G}$ extends $\vec{C}$.
Let $Q$ be any path in $G$ from a vertex $v_1\in V(C_1)$ to a vertex $v_2\in V(C_2)$ intersecting $C_1\cup C_2$ only in its endvertices,
and let $m=(\psi(v_2)-\psi(v_1))\bmod 3$.
Let $d:V(H)\to\mathbb{Z}$ be the function corresponding to $\dual{d}$.
By Lemma~\ref{tutte-flow}, it suffices to show that there exists a $d$-flow $h$ in $H$ with $\supp(h)=E(H)$ such that the restriction of
$\dual{h}$ to $E\bigl(\vec{C}\bigr)$ is equal to $\delta_{\vec{C},\psi}$ and $\int_Q \dual{h}\equiv m\pmod 3$.

For $i\in\{1,2\}$, let $z_i=\dual{f_i}$, and let $g_1$ and $g_2$ be the faces of $H-\{z_1,z_2\}$ such that $z_i$ is drawn in $g_i$ for $i\in\{1,2\}$
(since the cycles $C_1$ and $C_2$ are vertex-disjoint, we have $g_1\neq g_2$).
By Observation~\ref{obs-contrflow}, it suffices to show there exists
a $(d/\delta_{\vec{C},\psi})$-flow $h'$ in $H-\{z_1,z_2\}$ such that $\supp(h')=E(H-\{z_1,z_2\})$
and $\int_{Q} \dual{h'}\equiv m\pmod 3$.  By Corollary~\ref{alg-minmax}, it suffices to verify that
\begin{equation}\label{eq-check1}
\gain{d/\delta_{\vec{C},\psi}}(K)\ge 0
\end{equation}
for every cycle $K$ in $\dual{(H-\{z_1,z_2\})}$, and that
\begin{equation}\label{eq-check2}
\gain{d/\delta_{\vec{C},\psi}}(X)>2
\end{equation}
for every $(g_1,g_2)$-connecting set $X$ of cycles and edges of $\dual{(H-\{z_1,z_2\})}$.

Consider any cycle $K$ in $\dual{(H-\{z_1,z_2\})}$, and let $T$ be the subgraph of $G$ with $E(T)=E(K)$
and $V(T)$ consisting of the vertices incident with these edges.  If $\dual{g_1},\dual{g_2}\not\in V(K)$, then $T$
is a cycle in $G$ vertex-disjoint from $C_1\cup C_2$.  If $\dual{g_i}\in V(K)$ and $\dual{g_{3-i}}\not\in V(K)$ for some $i\in\{1,2\}$,
then $T$ is either a cycle in $G$ intersecting $C_i$ in one vertex and disjoint from $C_{3-i}$ (and thus edge-disjoint from
$C_1\cup C_2$), or $T$ is a generalized chord of $C_i$ vertex-disjoint from $C_{3-i}$.  Finally, if $\dual{g_1},\dual{g_2}\in V(K)$, then $T$
is a cycle intersecting each of $C_1$ and $C_2$ in one vertex, or a generalized chord of $C_i$ intersecting $C_{3-i}$ in one vertex
for some $i\in\{1,2\}$, or a $(C_1,C_2)$-connector.  In either case, $T$ is a constraint.  Observe that by the definition
of $d/\delta_{\vec{C},\psi}$, we have $$\gain{\dual{d},\psi}(T)=\gain{d/\delta_{\vec{C},\psi}}(K).$$
Furthermore, if $X$ is a $(g_1,g_2)$-connecting set $X$ of cycles and edges of $\dual{(H-\{z_1,z_2\})}$ and $X'$ is obtained
from $X$ by transforming each cycle as described above and keeping the edges of $X$ that are non-chord in $G$, then $X'$
is $(C_1,C_2)$-connecting.  Therefore, the inequalities (\ref{eq-check1}) and (\ref{eq-check2}) follow from the assumptions of this Lemma.
\end{proof}

Let us remark that while the condition that $\gain{\dual{d},\psi}(T)\ge 0$ for every constraint $T$ is also necessary for the extendability
of $\psi$, the condition on set of constraints is only sufficient: a coloring can extend even
if $\gain{\dual{d},\psi}(X)\le 2$ for some $(C_1,C_2)$-connecting set $X$ of constraints.

\section{Coloring graphs in the torus}\label{sec-tor}

Let us now apply the results to colorings of graphs in the torus. 
Suppose $G$ is a graph with a $2$-cell drawing in the torus,
and let $C$ be a non-contractible cycle in $G$.  Cutting the torus along $C$ results in a graph $G_0$ drawn in a cylinder $\Sigma$, with
the boundary tracing two cycles $C_1$ and $C_2$ corresponding to $C$.
Let $\gamma$ be the function mapping the points of $\Sigma$ to the corresponding points of the torus (note that $\gamma(C_1)=\gamma(C_2)=C$).
We can take $\Sigma$ as a part of the plane and obtain a plane drawing
of $G_0$ such that $C_2$ bounds the outer face.
We say that $(G_0,C_1,C_2,\gamma)$ is obtained from $G$ by \emph{cutting along $C$}.  Clearly, colorings of $G$ and $G_0$ are related
as follows.

\begin{observation}\label{obs-cut}
Let $G$ be a graph with a $2$-cell drawing in the torus, let $C$ be a non-contractible cycle in $G$, and let
$(G_0,C_1,C_2,\gamma)$ be obtained from $G$ by cutting along $C$.  Then $G$ is $3$-colorable if and only if there
exists a $3$-coloring $\psi$ of $C$ such that the $3$-coloring $\gamma\circ\psi$ of $C_1\cup C_2$ extends to a $3$-coloring of $G_0$.
\end{observation}

We now aim to prove Lemma~\ref{lemma-ew} that no 4-critical triangle-free graph can be drawn in the torus with edge-width
at least $21$.  Suppose for a contradiction such a graph $G$ exists, and let $C$ be a shortest non-contractible cycle in $G$.
By Observation~\ref{obs-cut}, to obtain a contradiction it suffices to cut $G$ along $C$ and find a $3$-coloring $\psi$ of $C$ such
that the corresponding coloring of $C_1\cup C_2$ extends to a $3$-coloring of $G_0$. To this end, we will choose $\psi$
and a function $\dual{d}:F(G_0)\to \mathbb{Z}$ carefully and apply Lemma~\ref{lemma-sufficient}.

The following result will be useful when bounding $\gain{\dual{d},\psi}$ for various constraints (recall that $G$ has only
very few faces of length other than $4$ by Theorem~\ref{thm-dvopek}).
For a graph $G$ with a $2$-cell drawing in a surface other than the sphere
and a contractible cycle $C$ in $G$, let $\dual{\inter}(C)$ denote the set of faces of $G$ drawn in the open disk
bounded by $C$.  For a set of faces $T$, let $S(T)$ denote the multiset of the lengths of the faces in $T$
of length other than $4$.

\begin{lemma}[Combination of Lemma~5.3 of~\cite{trfree4} and Theorem~3.1 of~\cite{thom-torus}]\label{lemma-long}
Let $G$ be a triangle-free graph with a $2$-cell drawing in a surface other than the sphere and let $C$ be a contractible
cycle in $G$ not bounding a face.  If $G$ is $4$-critical, then the following claims hold.
\begin{itemize}
\item $|C|\ge 6$,
\item if $S(\dual{\inter}(C))\neq\emptyset$, then $|C|\ge 7$,
\item if $|S(\dual{\inter}(C))|\ge 2$, then $|C|\ge 8$, and
\item if $|S(\dual{\inter}(C))|\ge 3$, then $|C|\ge 9$.
\end{itemize}
\end{lemma}

Next, we give several properties of constraints formed by generalized chords.
Let $(G_0,C_1,C_2,\gamma)$ be obtained from a $4$-critical triangle-free graph on torus by cutting along a non-contractible
cycle, let $U$ be a (possibly empty) set of odd-length faces of $G_0$, let $k$ and $t$ be positive integers, and let $B$ be a subpath of $C_i$
for some $i\in\{1,2\}$.  We say that $U$ is \emph{$(t,k)$-tied to (the subpath $B$ of) $C_i$} if $|E(B)|=t$ and $C_i$ has a generalized
chord $R$ of length $k$ with base $B$ such that $U$ is exactly the set of odd-length faces contained in $\dual{\inter}(R)$.
For a positive integer $n$, we say that $U$ is \emph{strongly $(\le\!n)$-tied to $C_i$} if $U$ is $(k,k)$-bound to $C_i$ for
some $k\le n$, and $U$ is \emph{$n$-loose with respect to $C_i$} otherwise.
A face $f$ of $G_0$ is \emph{$k$-near} to $C_i$ if $f\in \dual{\inter}(Q)$ for some generalized chord $Q$ of $C_i$ of length at most $k$.

\begin{lemma}\label{lemma-ties}
Let $G$ be a $4$-critical triangle-free graph with a $2$-cell drawing in the torus of edge-width at least $15$,
and let $C$ be a shortest non-contractible cycle
in $G$.  Let $(G_0,C_1,C_2,\gamma)$ be obtained from $G$ by cutting along $C$, and fix $i\in\{1,2\}$.
Let $U$ be a set of odd-length faces of $G_0$ that is $(t,k)$-tied to a subpath $Q$ of $C_i$.
Then the following claims hold.
\begin{itemize}
\item[(a)] $t\le k$
\item[(b)] If $|U|=2$ and $t<k\le 6$, then $(t,k)\in\{(3,5),(2,6),(4,6)\}$.
\item[(c)] If $|U|=2$, $t<k\le 6$, $U$ is also $(t',k')$-tied to another subpath $Q'$ of $C_i$, $t'<k'\le 6$, and $U$ is $7$-loose
with respect to $C_i$,
then $Q\subseteq Q'$ or $Q'\subseteq Q$, or $U$ is $(3,5)$-tied to either $Q\cup Q'$ or $Q\cap Q'$ (and in particular the union or intersection
is a path of length three).
\item[(d)] If $|U|=2$, $k=t\le 7$, $U'$ is a set of odd-length faces of $G_0$ such that $|U'|=2$, and $|U\cap U'|=1$,
then $U'$ is $7$-loose with respect to $C_i$.
\end{itemize}
\end{lemma}
\begin{proof}
Let $R$ be a generalized chord of $C_i$ of length $k$ with base $Q$ such that $U$ is the set of odd-length faces
of $\dual{\inter}(R)$.  Let $C'$ be the closed walk in $G$ obtained from $C$ by replacing
$\gamma(Q)$ by $\gamma(R)$.  Note that $C'$ is homotopically equivalent to $C$, and thus $C'$ is non-contractible.
Since $C$ is a shortest non-contractible cycle in $G$, we have $|C|-t+k=|C'|\ge |C|$, and thus $t\le k$.  Consequently, (a) holds.

In the cases (b), (c), and (d), $\dual{\inter}(R)$ contains exactly two odd-length faces, and thus $k+t$ is even.
Furthermore, by Lemma~\ref{lemma-long}, $k+t\ge 8$.  Hence, if $t<k\le 6$, then $(t,k)\in\{(3,5),(2,6),(4,6)\}$.
Therefore, (b) holds.

For the cases (c) and (d), let $R'$ be a generalized chord of $C_i$ of length $k'$ with base $Q'$ of length $t'$, with $U'$ being the set of odd-length faces
of $\dual{\inter}(R')$, where
\begin{itemize}
\item in the case (c) $U=U'$, and 
\item in the case (d) we for contradiction assume $t'=k'\le 7$.
\end{itemize}
Let $d=k-t$ and $d'=k'-t'$, so that $d=d'=0$ in case (d) and $d,d'\in\{2,4\}$ in case (c) by (b).
Let $f_1$ and $f_2$ be the faces of $G_0$ bounded by $C_1$ and $C_2$.
Let $\Delta$ and $\Delta'$ be the disks in the plane bounded by $R\cup Q$ and $R'\cup Q'$, respectively.
Consider the plane graph $G'=C_i\cup R\cup R'$. Let $D_0$ denote the set of faces of $G'$ not contained in $\Delta\cup \Delta'$,
and let $D_2$ denote the set of faces of $G'$ contained in $\Delta\cap \Delta'$; let $D=D_0\cup D_2$.
Note that if an edge $e\in E(R\cup R')$ is in the boundaries of two faces of $D$, then $e\in E(R\cap R')$.  Furthermore, an edge $e$ of $C_i$
is incident with a face of $D$ distinct from $f_i$ if and only if $e$ is contained
in either both or neither of $Q$ and $Q'$.
Hence,
\begin{equation}\label{eq-lbint}
\sum_{f\in D}|f|\le k+k'+2|C|-(t-|E(Q\cap Q')|)-(t'-|E(Q\cap Q')|)=2|C|+d+d'+2|E(Q\cap Q')|.
\end{equation}
Let $f_0$ denote the face of $G'$ containing $f_{3-i}$;
clearly, $f_0\in D_0$, and since $C$ is a shortest non-contractible cycle in $G$, we have $|f_0|\ge |C|\ge 15>k+k'$.
Hence, the boundary of $f_0$ intersects $C_i$ in at least one edge.

Let $\ell_2=\sum_{f\in D_2} |f|$.  If $E(Q\cap Q')$ is non-empty, then the face of $G'$ distinct from $f_i$ whose boundary contains $E(Q\cap Q')$
belongs to $D_2$ and $\ell_2\ge 2|E(Q\cap Q')|$ by (a).
If $E(Q\cap Q')=\emptyset$, then $\ell_2\ge 2|E(Q\cap Q')|$ trivially.  Furthermore,
the odd-length faces of $G$ drawn in $\Delta\cap \Delta'$ are exactly those belonging to $U\cap U'$, and thus
$\ell_2$ and $|U\cap U'|$ have the same parity.  Hence, $\ell_2$ is odd in the case (d), implying
$\ell_2>2|E(Q\cap Q')|$, and since $d=d'=0$, (\ref{eq-lbint}) gives
$$2|C|+2|E(Q\cap Q')|\ge \sum_{f\in D}|f|\ge |f_i|+|f_0|+\ell_2> 2|C|+2|E(Q\cap Q')|,$$
which is a contradiction.  Therefore, (d) holds.

From now on, we assume the case (c).
Then two odd-length faces of $G$ are drawn in $\Delta\cap \Delta'$, and Lemma~\ref{lemma-long} implies $\ell_2\ge 8$.
We can assume that $Q\not\subseteq Q'$ and $Q'\not\subseteq Q$, and thus $|E(Q\cap Q')|\le \min(t,t')-1\le 3$.
Let $\ell_0=\sum_{f\in D_0\setminus \{f_i,f_0\}} |f|$.
By (\ref{eq-lbint}), we have
\begin{equation}\label{eq-cmp}
2|C|+\ell_0+8\le |f_i|+|f_0|+\ell_0+\ell_2=\sum_{f\in D}|f|\le 2|C|+d+d'+2|E(Q\cap Q')|.
\end{equation}
If $D_0\neq\{f_i,f_0\}$, then $\ell_0\ge 4$, and (\ref{eq-cmp}) implies $d+d'+2|E(Q\cap Q')|\ge 12$.
Since $|E(Q\cap Q')|\le 3$, it follows that $\max(d,d')\ge 3$, and thus $\min(t,t')=2$.
But then $|E(Q\cap Q')|\le \min(t,t')-1=1$, and the same argument gives $\max(d,d')\ge 5$, which is a contradiction.
We conclude that $D_0=\{f_i,f_0\}$, and consequently $V(Q\cap Q')\neq \emptyset$, as otherwise $C_i$ has
two non-empty subpaths edge-disjoint from $Q\cup Q'$, each incident with a distinct face of $D_0\setminus\{f_i\}$.
Since the boundary of $f_0$ contains an edge of $C_i$, we conclude both $Q\cup Q'$ and $Q\cap Q'$ are non-empty connected
subpaths of $C_i$ (with $Q\cap Q'$ possibly consisting of a single vertex).
Let $T$ denote the part of the boundary of $f_0$ edge-disjoint from $C_i$.  Since $\Delta$ and $\Delta'$ are not disjoint,
observe that $T$ is a path intersecting $C_i$ only in its endpoints; hence, $T$ is a generalized chord of $C_i$ with base $Q\cup Q'$,
and $\dual{\inter}(T)=U$.

We claim that $|E(Q\cup Q')|\le 7$.  Indeed, since $t,t'\le 4$, we could have $|E(Q\cup Q')|>7$ only if $t=t'=4$
and $E(Q)\cap E(Q')=\emptyset$; but then $d,d'=2$, contradicting (\ref{eq-cmp}).
Since $U$ is $7$-loose with respect to $C_i$, it follows that $|E(T)|>|E(Q\cup Q')|$.
Furthermore, since $U$ consists of the odd-length faces in $\dual{\inter}(T)$ and $|U|$ is even,
$|E(T)|$ and $|E(Q\cup Q')|$ have the same parity, and thus $|E(T)|\ge |E(Q\cup Q')|+2$.
Note that $|f_0|=|C|+|E(T)|-|E(Q_1\cup Q_2)|$, and thus as in (\ref{eq-cmp}), we have
\begin{align}
2|C|+10&\le 2|C|+|E(T)|-|E(Q_1\cup Q_2)|+8\le |f_i|+|f_0|+\ell_2\nonumber\\
&\le \sum_{f\in D}|f|\le 2|C|+d+d'+2|E(Q\cap Q')|\label{eq-cmp2}.
\end{align}
Since $d,d'\le 4$, it follows that $|E(Q\cap Q')|\ge 1$.

If $|E(Q\cap Q')|=1$, then (\ref{eq-cmp2}) shows that $d=d'=4$ (and thus $t=t'=2$ and $|E(Q\cup Q')|=3$)
and $|E(T)|=|E(Q\cup Q')|+2=5$, implying that $U$ is $(3,5)$-tied to $Q\cup Q'$.  
If $|E(Q\cap Q')|\ge 2$, then $t,t'\ge 3$, and thus $d,d'=2$.  Hence, (\ref{eq-cmp2}) shows that $|E(Q\cap Q')|\ge 3$.
Since $t,t'\le 4$, we have $|E(Q\cap Q')|=3$, and thus all the inequalities in (\ref{eq-cmp2}) are tight; in particular $\ell_2=8$.
Considering the face of $G'$ distinct from $f_i$ whose
boundary contains $Q\cap Q'$, we conclude that $U$ is $(3,5)$-tied to $Q\cap Q'$.
\end{proof}

Analogously, we can prove the following bound
\begin{lemma}\label{lemma-bities}
Let $G$ be a $4$-critical triangle-free graph with a $2$-cell drawing in the torus,
and let $C$ be a shortest non-contractible cycle
in $G$.  Let $(G_0,C_1,C_2,\gamma)$ be obtained from $G$ by cutting along $C$, 
let $f_1$, \ldots, $f_4$ be distinct odd-length faces of $G$, and fix $i\in\{1,2\}$.
Suppose that for $j\in\{1,2\}$, $\{f_{2j-1},f_{2j}\}$ is $(t_j,k_j)$-tied to a subpath $Q_j$ of $C_i$ and $t_j<k_j$.
Then $|E(Q_1\cap Q_2)|\le (t_1+t_2+k_1+k_2)/2-8\le k_1+k_2-10$.
\end{lemma}
\begin{proof}
For $j\in\{1,2\}$, let $R_j$ denote a generalized chord of $C_i$ of length $k_j$ with base $Q_j$ and
with $\{f_{2j-1},f_{2j}\}$ being exactly the odd-length faces in $\dual{\inter}(R_j)$, let $\Delta_j$ denote
the open disk bounded by $Q_j\cup R_j$, let $D_j$ denote the set of faces of $G'=C_i\cup R_1\cup R_2$
drawn in $\Delta_j\setminus\Delta_{3-j}$, and let $\ell_j$ be the sum of the lengths of these faces.
Since $f_{2j-1}$ and $f_{2j}$ are contained in $\Delta_j\setminus\Delta_{3-j}$, Lemma~\ref{lemma-long} implies $\ell_j\ge 8$.
Note that if an edge $e$ is in boundaries of two faces of $D_1\cup D_2$, then $e\in E(R_1\cap R_2)$.  Furthermore, an edge $e$ of $C_i$
is incident with a face of $D$ if and only if $e$ is contained in exactly one of $Q_1$ and $Q_2$.
Hence,
$$16\le \ell_1+\ell_2\le k_1+k_2+|E(Q_1)|+|E(Q_2)|-2|E(Q_1\cap Q_2)|.$$
Note that $t_j=|E(Q_j)|$ and $k_j$ have the same parity, since $\dual{\inter}(R_j)$ contains exactly two odd-length faces,
and thus $t_j\le k_j-2$.  This implies the inequality from the statement of the lemma.
\end{proof}

Let $G$ be a $4$-critical triangle-free graph with a $2$-cell drawing in the torus,
and let $C$ be a non-contractible cycle in $G$.  Let $(G_0,C_1,C_2,\gamma)$ be obtained from $G$ by cutting along $C$,
and let $f_1$ and $f_2$ be the faces of $G_0$ bounded by $C_1$ and $C_2$.
We now select a function $\dual{d}:F(G_0)\to \mathbb{Z}$ so that the sources and sinks are in absolute value as small
as possible and pairs of faces strongly $(\le\!7)$-tied to $C_1$ or $C_2$ are assigned opposite values.
More precisely, we say that $\dual{d}$ is a \emph{standard assignment of sources and sinks} if $\dual{d}(F(G_0))=0$, $\dual{d}(f)=0$ for every even-length
face $f\in F(G_0)$, $\dual{d}(f)\in\{-3,3\}$ for every odd-length face $f\in F(G_0)$,
$\dual{d}(f_1)\in\{0,3\}$ and $\dual{d}(f_2)=-\dual{d}(f_1)\in\{0,-3\}$, and
$\dual{d}(U)=0$ for every two-element set $U$ strongly $(\le\!7)$-tied to $C_1$ or $C_2$.
Let us remark that the last condition is useful in arguing $\gain{\dual{d},\psi}(R)\ge 0$ for generalized chords $R$ with bases of length $|R|$.
Note also that a standard assignment is even, $3|\dual{d}$ and $|\dual{d}(f)|\le |f|$ for every $f\in F(G_0)$.

\begin{lemma}\label{lemma-exd}
Let $G$ be a $4$-critical triangle-free graph with a $2$-cell drawing in the torus of edge-width at least $15$,
and let $C$ be a shortest non-contractible cycle in $G$.  Let $(G_0,C_1,C_2,\gamma)$ be obtained from $G$ by cutting along $C$.
Then there exists a standard assignment of sources and sinks $\dual{d}:F(G_0)\to\mathbb{Z}$.
\end{lemma}
\begin{proof}
Let $f_1$ and $f_2$ be the faces of $G_0$ bounded by $C_1$ and $C_2$.
The values of $\dual{d}$ are exactly determined by the definition of a standard assignment except on
odd faces distinct from $f_1$ and $f_2$.  We need to choose the values on these odd faces so that
$\dual{d}(F(G_0))=0$ and the last condition from the definition holds.

Note that $G_0$ does not contain a path $P$ of length less than $|C|/2$ between $C_1$ and $C_2$, as otherwise one of the three
non-contractible cycles in $C\cup \gamma(P)$ is shorter than $C$.
If $R_i$ is a generalized chord of $C_i$ for $i\in\{1,2\}$ and $\dual{\inter}(R_1)\cap \dual{\inter}(R_2)\neq \emptyset$, then
$R_1\cup R_2$ contains two edge-disjoint paths from $C_1$ to $C_2$, and thus $(|E(R_1)|+|E(R_2)|)/2\ge |C|/2$.  Since $|C|\ge 15$,
we conclude that each face of $G_0$ is $7$-near to at most one of $C_1$ and $C_2$.  Let $D$ be an auxiliary graph whose
vertices are the odd-length faces of $G_0$ distinct from $f_1$ and $f_2$ and $ff'\in E(D)$ if and only if $\{f,f'\}$ is strongly
$(\le\!7)$-tied to $C_1$ or $C_2$.  Clearly, $|V(D)|$ is even.  By Lemma~\ref{lemma-ties}(d), $D$ has maximum degree at most~$1$, and thus $D$ is a partial matching.
For each edge $ff'$ of $D$, set $\dual{d}(f)=3$ and $\dual{d}(f')=-3$.  Finally, choose $\dual{d}(f)\in\{-3,3\}$ for every
isolated vertex $f$ of $D$ so that $\dual{d}(F(G_0))=0$; this is possible, since the number of these vertices is even.
\end{proof}

Let $\vec{C}$ be an orientation of a cycle $C$ and let $\psi$ be a $3$-coloring of $C$.
In order to maximize $\gain{\dual{d},\psi}$ for generalized chords, it is convenient if $\Bigl|\int_Q \delta_{\vec{C},\psi}\Bigr|$ is
relatively small for all subpaths $Q$ of $C$.
We say that $\psi$ is \emph{tame} if $\Bigl|\int_Q \delta_{\vec{C},\psi}\Bigr|\le 2$ for
every subpath $Q$ of $C$ of length at most $5$.
Since $\int_Q \delta_{\vec{C},\psi}$ and $|E(Q)|$ have the same parity,
if $|E(Q)|\in\{1,3,5\}$, then $\Bigl|\int_Q \delta_{\vec{C},\psi}\Bigr|=1$.
Any longer subpath $Q$ of $C$ can be partitioned into paths of length~$5$
and one path of length $|E(Q)|\bmod 5$, giving us the following bound.
\begin{observation}\label{lemma-tamevalues}
Let $\psi:V(C)\to \{0,1,2\}$ be a coloring of a cycle $C$, let $\vec{C}$ be an orientation of $C$,
and let $Q$ be a subpath of $C$.  If $\psi$ is tame, then
$$\Bigl|\int_Q \delta_{\vec{C},\psi}\Bigr|\le \lfloor |E(Q)|/5\rfloor+m(|E(Q)|\bmod 5),$$
where
$m(0)=0$, $m(1)=m(3)=1$, and $m(2)=m(4)=2$.
\end{observation}

Let $G$ be a $4$-critical triangle-free graph with a $2$-cell drawing in the torus,
and let $C$ be a non-contractible cycle in $G$.  Let $(G_0,C_1,C_2,\gamma)$ be obtained from $G$ by cutting along $C$.
Let $\dual{d}$ be a standard assignment of sources and sinks.
Let $f_1$ and $f_2$ be the faces of $G_0$ bounded by $C_1$ and $C_2$.
Let $\vec{C}$ be a cyclic orientation of $C$ chosen so that the paths in $C_1$ for which $f_1$ is to their right
are mapped by $\gamma$ to paths along $\vec{C}$ (and thus, the paths in $C_2$ for which $f_2$ is to their right
are mapped by $\gamma$ to paths in the opposite direction to $\vec{C}$).  
The \emph{$(G,C,\dual{d})$-request} is the system of all pairs $(Q,s)$, where $Q$ is a subpath of $C$ directed along $\vec{C}$
and $s$ is a non-zero integer, such that one of the following holds:
\begin{itemize}
\item[(Ra)] $G_0$ has a two-element set $U$ of odd faces such that $\dual{d}(U)\neq 0$, for some $i\in\{1,2\}$,
$U$ is $(3,5)$-tied to a subpath $Q'$ of $C_i$, $\gamma(Q')$ is (up to reversal) equal to $Q$, and
$s=(-1)^i\dual{d}(U)$; or
\item[(Rb)] $G_0$ has a two-element set $U$ of odd faces such that $\dual{d}(U)\neq 0$, for some $i\in\{1,2\}$,
$U$ is $(2,6)$- or $(4,6)$-tied to a subpath $Q'$ of $C_i$, $\gamma(Q')$ is (up to reversal) equal to $Q$, and
$s=(-1)^i\dual{d}(U)$; or
\item[(Rc)] $G_0$ has a $5$-face $f$ such that for some $i\in\{1,2\}$, the boundary of $f$ intersects $C_i$ in a path
$Q'$ of length two, $\gamma(Q')$ is (up to reversal) equal to $Q$, and
$s=(-1)^i\dual{d}(f)$.
\end{itemize}
The $(G,C,\dual{d})$-request $\RR$ is \emph{satisfied} by a $3$-coloring $\psi$ of $C$ if $\int_C \delta_{\vec{C},\psi}=\dual{d}(f_1)$
and for every $(Q,s)\in \RR$, $$s\cdot\int_Q \delta_{\vec{C},\psi}\ge 0.$$  Note this means that the amount of charge sent
across $Q'$ according to $\psi$ ``compensates'' for the charge originating in $U$ or $f$ according to $\dual{d}$.
It turns out this suffices to ensure extendability of $\psi$.

\begin{lemma}\label{lemma-satisfy}
Let $G$ be a $4$-critical triangle-free graph with a $2$-cell drawing in the torus of edge-width at least $21$,
let $C$ be a shortest non-contractible cycle in $G$, and let $\vec{C}$ be a cyclic orientation of $C$.
Let $(G_0,C_1,C_2,\gamma)$ be obtained from $G$ by cutting along $C$, let $\dual{d}$ be a standard assignment of sources and sinks,
and let $\psi_0$ be a $3$-coloring of $C$.  If $\psi_0$ is tame, then $\psi_0$ does not satisfy the $(G,C,\dual{d})$-request $\RR$.
\end{lemma}
\begin{proof}
Suppose for a contradiction that $\psi_0$ satisfies $\RR$.
Let $f_1$ and $f_2$ be the faces of $G_0$ bounded by $C_1$ and $C_2$.
By Theorem~\ref{thm-dvopek}, $G$ has at most four odd-length faces, and thus
$|\dual{d}(U)|\le 6$ for any $U\subseteq F(G_0)\setminus \{f_1,f_2\}$, and $|\dual{d}(U)|\le 9$ for any $U\subseteq F(G_0)$.
Let $\psi=\gamma\circ\psi_0$ be the $3$-coloring of $C_1\cup C_2$ corresponding to $\psi_0$,
and let $\delta=\delta_{\vec{C_1}\cup\vec{C_2},\psi}$ for arbitrary orientations of $C_1$ and $C_2$.
Let us discuss possible constraints.
\begin{itemize}
\item Let $K$ be a cycle in $G_0$ edge-disjoint from $C_1\cup C_2$.  If $K$ is $(f_1,f_2)$-non-contractible,
then $\gamma(K)$ is non-contractible, and thus $|K|\ge 21$ and $\gain{\dual{d},\psi}(K)=|K|-|\dual{d}(\dual{\inter}(K))|\ge 12$.
Suppose that $K$ is $(f_1,f_2)$-contractible.  If $|K|\le 7$, then by Lemma~\ref{lemma-long}
$\dual{\inter}(K)$ contains at most one odd-length face, and
$$\gain{\dual{d},\psi}(K)=|K|-|\dual{d}(\dual{\inter}(K))|=\begin{cases}|K|&\text{ if $|K|$ is even}\\|K|-3&\text{ if $|K|$ is odd.}\end{cases}$$
If $|K|\ge 8$, then $\gain{\dual{d},\psi}(K)=|K|-|\dual{d}(\dual{\inter}(K))|\ge |K|-6$.

We conclude
that $\gain{\dual{d},\psi}(K)\ge 2$, and $\gain{\dual{d},\psi}(K)\ge 4$ unless $K$ is $(f_1,f_2)$-contractible and either
$|K|=5$ and $\dual{\inter}(K)$ consists of a $5$-face, or $|K|=8$ and $\dual{\inter}(K)$ contains two odd-length faces.

\item Let $R$ be a $(C_1,C_2)$-connector consisting of paths $P_1$ and $P_2$.
Let $Q_1$ and $Q_2$ be subpaths of $C_1$ and $C_2$ directed so that $f_1$ and $f_2$ are to the right of them,
such that $K=Q_1\cup Q_2\cup P_1\cup P_2$ is a $(f_1,f_2)$-contractible cycle.
Note that there are two possible choices for $Q_1$ and $Q_2$, and we choose one where $|E(Q_1)|+|E(Q_2)|\le |C|$.
For $i\in\{1,2\}$, let $a_i=\int_{Q_i} \delta$.  Since $C$ is a shortest non-contractible cycle in $G$, we have $|E(P_1)|,|E(P_2)|\ge |C|/2$.
We have $\gain{\dual{d},\psi}(R)\ge |E(R)|-|a_1|-|a_2|-|\dual{\inter}(K)|\ge |C|-|a_1|-|a_2|-6$.
Since $\psi$ is tame, Observation~\ref{lemma-tamevalues} gives $|a_1|+|a_2|\le |E(Q_1)|/5+|E(Q_2)|/5+4\le |C|/5+4$, and thus
$\gain{\dual{d},\psi}(R)\ge \tfrac{4}{5}|C|-10>4$.

\item Let $R$ be a generalized chord of $C_i$ for some $i\in \{1,2\}$; by symmetry, we can assume $i=1$.
Let $Q$ be the base of $R$, and let $k=|E(R)|$, $t=|E(Q)|$ and $a=\int_Q \delta$.  Since $C$ is a shortest
non-contractible cycle, we have $k\ge t$ and $k\ge 2$.  Recall that $\gain{\dual{d},\psi}(R)=k-|a+\dual{d}(\dual{\inter}(R))|$.

\begin{itemize}
\item If $\dual{d}(\dual{\inter}(R))=0$, then $k$ and $t$ have the same parity, and by Observation~\ref{lemma-tamevalues},
we have $\gain{\dual{d},\psi}(R)\ge 0$ if $k=2$, $\gain{\dual{d},\psi}(R)\ge 2$ if
$k\in\{3,4\}$, and $\gain{\dual{d},\psi}(R)\ge 4$ if $k\ge 5$.

\item If $|\dual{d}(\dual{\inter}(R))|=3$, then by parity we have $k\ge t+1$, and since $G$ is triangle-free, $k\ge 3$.  If $k=3$, then
$R\cup Q$ is a $5$-cycle, and by Lemma~\ref{lemma-long} $R\cup Q$ bounds a $5$-face $f$.  Since $t=2$,
$(\gamma(Q),-\dual{d}(f))\in\RR$, and since $\psi$ satisfies $\RR$, $a$ and $\dual{d}(\dual{\inter}(R))$
do not have the same sign.
Consequently, $|a+\dual{d}(\dual{\inter}(R))|\le 3$ and $\gain{\dual{d},\psi}(R)\ge 0$.  

By Observation~\ref{lemma-tamevalues},
we have $\gain{\dual{d},\psi}(R)\ge 0$ if $k\in\{4,5\}$, $\gain{\dual{d},\psi}(R)\ge 2$ if $k\in\{6,7,8\}$, and
$\gain{\dual{d},\psi}(R)\ge 4$ if $k\ge 9$.

\item If $|\dual{d}(\dual{\inter}(R))|=6$, then $\dual{\inter}(R)$ contains exactly two odd-length faces $f_1$ and $f_2$ such that 
$\dual{d}(f_1)=\dual{d}(f_2)\in\{-3,3\}$.  Since $\dual{d}$ is a standard assignment, if $k\le 7$ then $k\ge t+2$.

If $k\le 5$, then by Lemma~\ref{lemma-ties}(b) we have $k=5$ and $t=3$.  Hence,
$(\gamma(Q),-\dual{d}(\dual{\inter}(R)))\in\RR$, and since $\psi$ satisfies $\RR$, $a$ and $\dual{d}(\dual{\inter}(R))$
have opposite signs.  Consequently $|a+\dual{d}(\dual{\inter}(R))|\le 5$ and $\gain{\dual{d},\psi}(R)\ge 0$.
If $k=6$, then an analogous argument gives $|a+\dual{d}(\dual{\inter}(R))|\le 6$ and $\gain{\dual{d},\psi}(R)\ge 0$.

By Observation~\ref{lemma-tamevalues},
we have $\gain{\dual{d},\psi}(R)\ge 0$ if $k\in\{7,8,9\}$, $\gain{\dual{d},\psi}(R)\ge 2$ if $k\in\{10,11,12\}$, and
$\gain{\dual{d},\psi}(R)\ge 4$ if $k\ge 13$.
\end{itemize}
\end{itemize}
Hence, $\gain{\dual{d},\psi}(R)\ge 0$ for every constraint $R$.  Since $G$ is $4$-critical,
$\psi_0$ does not extend to a $3$-coloring of $G$, and by Observation~\ref{obs-cut}, $\psi$ does not extend
to a $3$-coloring of $G_0$. By Lemma~\ref{lemma-sufficient}
we conclude that there exists a $(C_1,C_2)$-connecting set $X$ of constraints
such that $\gain{\dual{d},\psi}(X)\le 2$.  In particular, we have $\gain{\dual{d},\psi}(R)\le 2$ for every $R\in X$.
According to the preceding analysis, we conclude that $X$ contains for $i\in\{1,2\}$ either a non-chord edge $R_i$ with one
end in $C_i$ or a generalized chord $R_i$ of $C_i$, and in case $R_1$ and $R_2$ are generalized chords with
$\gain{\dual{d},\psi}(R_1)=\gain{\dual{d},\psi}(R_2)=0$, $X$ can additionally contain a non-chord edge $R_3$ or an $(f_1,f_2)$-contractible
cycle $R_3$ with $\gain{\dual{d},\psi}(R_3)=2$.
Note that $G_0$ contains a path from $C_1$ to $C_2$ of length at most $\sum_{R\in X} \lfloor |R|/2\rfloor$,
and by a straightforward case analysis using the description of constraints $R$ with $\gain{\dual{d},\psi}(R)\le 2$ we obtained above,
this path has length at most $10$.  However, since $C$ is a shortest non-contractible cycle in $G$, any such path must have length at
least $|C|/2>10$, which is a contradiction.
\end{proof}

We are now ready to prove that no $4$-critical triangle-free graph can be drawn in the torus with edge-width at least 21,
by constructing a $3$-coloring satisfying the request.
\begin{proof}[Proof of Lemma~\ref{lemma-ew}]
Suppose for a contradiction that $G$ is a $4$-critical triangle-free graph with a $2$-cell drawing in the torus of edge-width at least $21$,
let $C$ be a shortest non-contractible cycle in $G$.  Let $(G_0,\Sigma,C_1,C_2,\gamma)$ be obtained from $G$ by cutting along $C$.
Let $\vec{C}$ be a cyclic orientation of $C$ chosen so that the paths in $C_1$ for which $f_1$ is to their right
are mapped by $\gamma$ to paths along $\vec{C}$.

We say that an odd-length face $f\in F(G_0)\setminus\{f_1,f_2\}$
is \emph{isolated} if no two-element set containing $f$ is strongly $(\le\!7)$-tied to $C_1$ or $C_2$,
and that $f$ is \emph{aligned} if for some $i\in\{1,2\}$, $f$ is $7$-near to $C_i$ and $\dual{d}(f)=(-1)^i\cdot 3$.
Let $\dual{d}$ be a standard assignment of sources and sinks such that the number of aligned isolated faces is maximum.
Let $\RR$ be the $(G,C,\dual{d})$-request.

If $|C|$ is odd, let $\psi_0$ be a $3$-coloring of $C$ where vertices in order have colors
\begin{equation}\label{col-odd}
\mathbf{0, 1, 2}, 1, 2, \mathbf{1, 2, 0}, 2, 0, \mathbf{2, 0, 1}
\end{equation}
followed by $0, 1$ repeated $(|C|-13)/2$ times; if $|C|$ is even, let $\psi_0$ be a $3$-coloring
of $C$ where vertices in order have colors
\begin{equation}\label{col-even}
\mathbf{0, 1, 2}, 1, \mathbf{2, 1, 0}, 1
\end{equation}
followed by $0, 1$ repeated $(|C|-8)/2$ times.  The boldface emphasizes the places where $\delta_{\vec{C},\psi_0}(e_1)=\delta_{\vec{C},\psi_0}(e_2)$
for two consecutive edges $e_1$ and $e_2$ of $C$; clearly, $\psi_0$ is tame and $\int_C \delta_{\vec{C},\psi_0}=\dual{d}(f_1)$.
For $k\in \{1,\ldots, |C|-1\}$, let $\psi_k$ denote
the coloring of $C$ obtained by rotating $\psi_0$ on $C$ by $k$ vertices.  We say the index $k$ is \emph{killed} by $(Q,s)\in\RR$
if $s\cdot\int_Q \delta_{\vec{C},\psi_k}<0$.  By Lemma~\ref{lemma-satisfy}, $\RR$ is not satisfied, and thus every index $k\in\{0,\ldots,|C|-1\}$
is killed by some element of $\RR$.

Let $U^+$ and $U^-$ be the set of faces of $F(G_0)\setminus \{f_1,f_2\}$ to that $\dual{d}$ assigns the value $3$ and $-3$, respectively.
Recall that by Theorem~\ref{thm-dvopek}, $G$ has at most $4$ odd-length faces, and thus $|U^+|=|U^-|\le 2$.
For $i\in \{1,2\}$, let $U^+_i$ denote those of the faces in $U^+$ that are $7$-near to $C_i$; $U^-_i$ is defined analogously.
Let us now discuss the elements of $\RR$ arising from $U^+$.
Since $C$ is a shortest non-contractible cycle and $|C|\ge 21$, the sets $U^+_1$ and $U^+_2$ are disjoint.
\begin{itemize}
\item If $|U^+_1|\le 1$ and $|U^+_2|\le 1$, then $U^+$ only contributes to $\RR$ by (Rc): $U^+_1$ may contribute $(Q_1,-3)$
and $U^+_2$ may contribute $(Q_2,3)$ for some subpaths $Q_1,Q_2\subset C$ of length two.  In total the contributed elements
kill at most $3$ indices if $|C|$ is odd and at most two indices if $|C|$ is even.
\item Suppose $|U^+_1|=2$ (and thus $U^+_2=\emptyset$).  Note that $U^+_1$ is $7$-loose with respect to $C_1$, since $\dual{d}$ is
a standard assignment of sources and sinks.
\begin{itemize}
\item If $U^+_1$ is $(3,5)$-tied to a subpath $Q$ of $C_1$, then $U^+$ contributes to $\RR$ the element $(Q,-6)$, and
by Lemma~\ref{lemma-ties}(c) can additionally contribute only elements $(Q',s)$ where $s\in\{-3,-6\}$ and $Q'$ is either
a length-two subpath or a length-four superpath of $Q$; and observe that each index killed by $(Q',s)$ is also killed by $(Q,s)$.
Consequently, the elements contributed by $U^+$ kill $(|C|+9)/2$ indices if $|C|$ is odd and $|C|/2$ indices if $|C|$ is even.
\item If $U^+_1$ is not $(3,5)$-tied to $C_1$, but is $(4,6)$-tied to a subpath $Q$ of $C_1$, then
$U^+$ contributes to $\RR$ the element $(Q,-6)$, and by Lemma~\ref{lemma-ties}(b) and (c) can additionally contribute
only elements $(Q',s)$ where $s\in\{-3,-6\}$ and $Q'$ is a length-two subpath of $Q$.  Note that each index killed by $(Q,-6)$
is also killed by $(Q',-3)$ for some length-two subpath of $Q$, and by considering three consecutive length-two subpaths
of $C$, we conclude the elements contributed by $U^+$ kill at most $9$ indices if $|C|$ is odd and at most $3$ indices if $|C|$ is even.
\item If $U^+_1$ is neither $(3,5)$-tied nor $(4,6)$-tied to $C_1$, then by Lemma~\ref{lemma-ties}(b) and (c) contributes
to $\RR$ either one element $(Q_1,-6)$ for a length-two subpath $Q_1$, or (by (Rc)) two elements $(Q_1,-3)$ and $(Q_2,-3)$
for length-two subpaths $Q_1$ and $Q_2$.  Consequently, the elements contributed by $U^+$ kill at most $6$ indices if $|C|$ is odd
and at most $2$ indices if $|C|$ is even.
\end{itemize}
\item The case that $|U^+_2|=2$ is symmetric, but the signs are switched; hence, in the three considered subcases,
the elements contributed by $U^+$ kill
\begin{itemize}
\item $(|C|-9)/2$ indices if $|C|$ is odd and $|C|/2$ indices if $|C|$ is even,
\item no indices if $|C|$ is odd and at most $3$ indices if $|C|$ is even, and
\item no indices if $|C|$ is odd and at most $2$ indices if $|C|$ is even, respectively.
\end{itemize}
\end{itemize}
The situation for $U^-$ is symmetric, up to switching of signs.  Let us first consider the case that $|C|$ is even.  Since all indices
are killed and $|C|/2>3$, this implies each of $U^+$ and $U^-$ is $(3,5)$ tied to one of $C_1$ and $C_2$, and by the preceding
analysis, each index is killed by exactly one of elements $(Q^+,s^+)$ and $(Q^-,s^-)$ of $\RR$, where $Q^+$ and $Q^-$ are length-three subpaths of $C$
and $s^+,s^-\in \{-6,6\}$.  Without loss of generality, we can assume that the index $0$ is killed by $(Q^+,s^+)$ and in $\psi_0$, the path $Q^+$
covers the first four colors of (\ref{col-even}); all other cases are symmetric.  Then also $1$, $|C|-1$, and $|C|-2$ are killed by $(Q^+,s^+)$,
and thus none of them is killed by $(Q^-,s^-)$.  This is only possible if $Q^-$ covers the last four colors of (\ref{col-even}).
However, then $2$ is killed by neither of the elements, which is a contradiction.

Therefore, $|C|$ is odd.  We can by symmetry assume that at least $|C|/2>9$ indices are killed by elements contributed by $U^+$,
and thus $|U^+|=2$ and $U^+$ is $(3,5)$-tied to $C_1$.  If $|U^-_2|=2$, then since $\dual{d}$ is a standard assignment of sources and sinks,
all elements of $U^+\cup U^-$ are isolated but not aligned, and there exists a standard assignment of sources and sinks
with more isolated aligned faces, contradicting the choice of $\dual{d}$.  Hence, $|U^-_2|\le 1$.  Since all indices are killed,
the elements contributed by $U^-$ kill at least $(|C|-9)/2>3$ indices.  It follows that $|U^-_1|=2$ and $U^-$
is $(3,5)$-tied to $C_1$.  By the preceding analysis, each index is killed by exactly one of elements $(Q^+,-6)$ and $(Q^-,6)$ of $\RR$,
where $Q^+$ and $Q^-$ are length-three subpaths of $C$; and by Lemma~\ref{lemma-bities}, $Q^+$ and $Q^-$ are edge-disjoint.
Without loss of generality, we can assume that in $\psi_0$, the path $Q^+$
starts two vertices before the first color of (\ref{col-odd}), and thus $(Q^+,-6)$ kills $0$, $1$, $2$, and $3$.  Hence, $(Q^-,6)$ does not kill
$0$, $1$, $2$, and $3$, which is only possible if $Q^-$ is shifted by $5$ or $10$ vertices to the right from $Q^+$ in (\ref{col-odd}).
However, then $|C|-2$ is not killed by either of the elements, which is a contradiction.
\end{proof}

\bibliographystyle{siam}
\bibliography{../data.bib}

\end{document}